\documentclass[11pt]{amsart}
\usepackage{amsfonts,amsmath,amssymb,fullpage,graphicx, comment}
\usepackage{verbatim}
\usepackage{xcolor}
\usepackage{hyperref}

\usepackage{array}
\usepackage{booktabs}

\usepackage{bbm}
\usepackage{bm}
\usepackage{dsfont}

\newtheorem{hyp}{Hypothesis}

\newtheorem{thm}{Theorem}[section]

\newtheorem{lem}[thm]{Lemma}
\newtheorem{example}[thm]{Example}
\newtheorem{prop}[thm]{Proposition}
\theoremstyle{definition}
\newtheorem{defn}[thm]{Definition}
\theoremstyle{remark}
\newtheorem{rem}[thm]{\bf{Remark}}
\numberwithin{equation}{section}

\newcommand{\beas}{\begin{eqnarray*}}
	\newcommand{\eeas}{\end{eqnarray*}}
\newcommand{\bes} {\begin{equation*}}
	\newcommand{\ees} {\end{equation*}}
\newcommand{\be} {\begin{equation}}
	\newcommand{\ee} {\end{equation}}
\newcommand{\bea} {\begin{eqnarray}}
	\newcommand{\eea} {\end{eqnarray}}

\title{Large  values of Hardy's Function from the Selberg class}

\begin{document}
	
	\author{Ahammad Mostafa Hossain}
	\address{Ahammad Mostafa Hossain\\ Department of Mathematics \\
		Indian Institute of Technology Kharagpur \\
		Kharagpur-721302,  India.} 
	\email{mostafa.24@kgpian.iitkgp.ac.in}

	\author{Kamalakshya Mahatab}
	\address{Kamalakshya Mahatab\\ Department of Mathematics \\
		Indian Institute of Technology Kharagpur \\
		Kharagpur-721302,  India.} 
	\email{kamalakshya@maths.iitkgp.ac.in}

	\author{Akshaa Vatwani}
	\address{Akshaa Vatwani\\ Department of Mathematics \\
		Indian Institute of Technology Gandhinagar, \\
		Palaj,	Gandhinagar Gujarat-382355,  India.} 
	\email{akshaa.vatwani@iitgn.ac.in}    
	
	%\thanks{}	
	

\begin{abstract}
	Let  $Z_F(t):=F(1/2+it) H(1/2+it)^{-1/2}$ denote Hardy's function, where the function $F(s)$  belongs to the Selberg  class with polynomial Euler product and satisfies the functional equation $F(s)=H(s)\overline{F}(1-\bar{s})$. We prove that for every function  of   degree $1$ and for certain functions of degree $2$  belonging to the Selberg class, one has 
	
	\begin{align*}
		\max_{T\leq t\leq 2T} \pm Z_F(t) &\gg \exp\left(C_F \sqrt{ \frac{{\log T}}{{\log \log T}}  } \right),
	\end{align*}
	where the constant $C_F$ depends on $F(s)$ and are explicitly determined.
\end{abstract}

	\maketitle
	%\tableofcontents
	%\newpage
	\section{Introduction}
	The study of the Riemann zeta function $\zeta(s)$ plays an important role in understanding the distribution of the prime numbers. At the heart of this lies the important observation of Riemann that the Prime Number Theorem is equivalent to showing that $\zeta(s)$ has no zeros on the line $\Re(s)=1$. The well-known Riemann Hypothesis in fact asserts that all the zeros of $\zeta(s)$ in the critical strip $0<\Re(s)<1$ lie on the line $\Re(s) =\frac12$.	
	One of the major advances towards this conjecture was made by Hardy \cite{hardy1914zeros} who showed that $\zeta(s)$  has infinitely many zeros on the critical line $\Re(s)=\frac12$. A key point in the proof of this results is the introduction of Hardy's function for $\zeta(s)$: a real-valued function $Z(t)$ of the real valuable of $t$, whose zeros are exactly the zeros of $\zeta(s)$ on the critical line $\Re(s) =\frac12$. A similar argument applies to Dirichlet $L$-functions $L(s,\chi)$.

	In a different direction, the Riemann zeta function can be viewed as a prototype of a more general class of $L$-functions. Chandrasekharan and  Narasimhan \cite{chandrasekharan1968zeta} showed that the zeta functions of ideal classes in quadratic fields possess infinitely many zeros on the critical line by constructing a suitable analogue of Hardy's function. Hafner~\cite{hafner1983,hafner1987} showed that the $L$-functions
	attached to holomorphic cusp forms and Maass wave forms for the full
	modular group with trivial character have infinitely many zeros on the
	critical line. More precisely,  he proved  that $N_o(T)\gg T \log T$, where $N_o(T)$ counts the zeros on the critical line with imaginary part in $(0,T]$. Later,  Rezvyakova \cite{rez2010} established the corresponding result for
	$L$-functions associated with automorphic cusp forms on Hecke
	congruence subgroups of arbitrary integral weight. See the necessary definitions in \cite{iwaniec1997topics}.
	
	Wilton~\cite{wilton1929note} studied the additively twisted
	$L$-function of the modular discriminant $\Delta(z)$, defined by
	\[
	\rho(s;p,q)=\sum_{n\geq 1}\frac{\tau(n)}{n^s}e^{2\pi i n p/q}.
	\]
	He proved that, for rational $p/q$ satisfying
	$p^2\equiv 1\pmod q$, this function has infinitely many zeros on the
	line $\Re(s)=6$, where $\tau(n)$ denotes the $n$-th Fourier coefficient
	of $\Delta(z)$. His proof used a variant  of Hardy's method together
	with the Hardy-type function
	\[
	H_{\alpha}(t,n)=t^nR(t;p,q)e^{\alpha t},
	\qquad
	R(t;p,q)=\left(\frac{q}{2\pi}\right)^{6+it}
	\Gamma(6+it)\rho(6+it;p,q),
	\]
	where $-\pi/2<\alpha < \pi/2$.
	Later, Kim~\cite{kim2023infinitely} generalized this result to cusp
	forms of integral and half-integral weights of higher level. In \cite{Meher2017pujahari}, Meher, Pujahari, and Kotyada studied the zeros
	of $L$-functions associated with cusp forms. For half-integral weight
	$k+1/2$ cusp forms of level $4$, they showed that infinitely many zeros
	lie on the critical line $\Re(s)=\frac{k+1/2}{2}$. They also obtained a
	similar result for integral weight $k$ cusp forms of level $\mathcal{N}$,
	where infinitely many zeros lie on the critical line
	$\Re(s)=\frac{k}{2}$. This result was later generalized by Meher, Pujahari, and
	Shankhadhar~\cite{Meher2019pujahari} to half-integral weight cusp forms
	of level $4\mathcal{N}^2$.

	Many of the $L$-functions considered above, such as those attached to integral weight cusp forms and Dedekind zeta functions, belong to a broader axiomatic framework known as the Selberg class $\mathbb{S}$, introduced by Selberg \cite{selberg1991old} in 1992. This class consists of functions satisfying a set of analytic properties, including a functional equation of the type that will be essential for constructing an appropriate analogue of Hardy's function. We now recall the precise definition of the Selberg class.
	\begin{defn}{(The Selberg Class).}
		The Selberg class $\mathbb{S}$ is the collection of meromorphic functions $F(s)$ for which the following conditions hold.
		\begin{enumerate}
			\item[(i)] \textbf{Dirichlet series} - The function $F(s)$ admits a Dirichlet series representation 
			$$
			F(s)=\sum_{n=1}^{\infty}\frac{a_F(n)}{n^s},
			$$
			which converges absolutely for $\Re(s)>1$. Also, the leading coefficient $a_F(1)$ is normalized as $a_F(1)=1$.
			
			\item[(ii)] \textbf{Analytic continuation} - For some integer $k\geq 0$, the function $(s-1)^kF(s)$ is entire and of finite order.

			\item[(iii)] \textbf{Functional equation} - Choose $Q>0$, $\alpha_i > 0$, $\beta_i\in\mathbb{C}$, and $\omega\in\mathbb{C}$ with $\Re(\beta_i)\geq 0$ and $|\omega|=1$. Set
			\begin{equation}\label{Phi define}
				\Phi(s):=Q^s\prod_{i=1}^{m}\Gamma(\alpha_i s+\beta_i)F(s).
			\end{equation}
			Then $\Phi(s)$ satisfies the functional equation
			\begin{equation}\label{functional equation}
				\Phi(s)=\omega \overline{\Phi}(1-\bar{s}).
			\end{equation}

			\item[(iv)] \textbf{Euler product} - For $\Re(s)>1$, the function $F(s)$ has the Euler product
			\begin{equation}\label{euler pro}
				F(s)=\prod_{p\,\text{prime}} F_p(s),
			\end{equation}
			where
			$$ \log F_p(s)=\sum_{k=1}^{\infty} \frac{b_{F}(p^k)}{p^{ks}}$$ with $b_{F}(p^k)=O(p^{k\theta})$ for some $\theta < \frac12.$

			\item[(v)]\textbf{Ramanujan hypothesis} - For any $\epsilon >0,$ $$a_F(n)=O_{\epsilon}(n^{\epsilon}).$$
		\end{enumerate}  
		We further use a stronger version of axiom (iv), known as the polynomial Euler product:

		\begin{enumerate}
			\item[$ (\mathrm{iv'})$]\textbf{Polynomial Euler product} - for $\Re(s)>1$  we have 
			\begin{equation}\label{poly euler pro}
				F(s)=\prod_{p\,\text{prime}} \prod_{i=1}^{r} \left( 1-\frac{\nu_i(p)}{p^s}\right)^{-1},
			\end{equation}
			where $\nu_i(p)$ are complex numbers.
		\end{enumerate}
		Let $\mathbb{S}_{\text{poly}}$ denote the subclass of $\mathbb{S}$ comprising all $L$-functions that satisfy axiom $(\mathrm{iv'})$.
		This additional assumption is not expected to exclude any important $L$-functions arising in number theory. In fact, all presently known examples in the Selberg class are believed to satisfy $(\mathrm{iv'})$, at least under certain widely accepted conjectures. For our purposes, the main consequence of $(\mathrm{iv'})$ is that, together with the Ramanujan hypothesis $(\mathrm{v})$, it gives $|\nu_i(p)| \leq 1$ (see \cite[Lemma 2.2]{steuding2007value}). Therefore, we have
		\begin{equation*}
			a_F(p)=b_F(p)=\sum_{i=1}^{r} \nu_i(p) \quad \text{and} \quad |a_F(p)|\leq r.
		\end{equation*}
		The Euler product structure gives the multiplicativity of the coefficients, namely,
		$a_F(mn)=a_F(m)a_F(n)$ whenever $(m,n)=1$. In addition, every Euler factor admits a Dirichlet series representation
		$$
		F_p(s)=\sum_{k=0}^{\infty}\frac{a_F(p^k)}{p^{ks}},
		$$
		and this series converges absolutely in the half-plane $\Re(s)>0$. Moreover, $F_p(s)$ has no zeros for $\Re(s)>\theta$, where $\theta$ is the quantity appearing in \eqref{euler pro}.
	\end{defn}

	The constants appearing in the functional equation \eqref{Phi define} depend on $F$. Since the duplication formula for the $\Gamma$-function may give different forms of the functional equation, these constants are not unique. However, the quantity
	$$
	d_F:=2\sum_{i=1}^{m}\alpha_i
	$$
	is a well-defined invariant called the degree of $F$.
	It is conjectured that $d_F$ is always an integer. Henceforth, we will denote the members of $\mathbb{S}$ having degree $d$ by $\mathbb{S}_d.$ We also define the $\xi$-invariant and  the conductor by
	\begin{equation}\label{zai invariant}
		\xi_F := 2 \sum_{j=1}^{m}  \left ( \beta_j-\frac{1}{2} \right) \quad \text{and} \qquad q_F:= (2 \pi )^{d_F} Q^2 \prod_{j=1}^{m} \alpha_{j}^{2 \alpha_j},
	\end{equation}
	respectively. %It is conjectured that $q_F$ is non-negative integer.
	
	Kaczorowski and Perelli~\cite{kaczorowski1999structure,kaczorowski2011structure}
	studied the Selberg class $\mathbb{S}$ in the degree range $0<d_F<2$ and
	obtained a complete description of the functions that occur in this range.
	They showed that the possible functions are the Riemann zeta function
	$\zeta(s)$, the Dirichlet $L$-functions $L(s,\chi)$ associated with
	primitive characters $\chi$, and their imaginary shifts
	$L(s+i\theta,\chi)$, where $\theta\in\mathbb{R}$.  Also note that for non-primitive character $\chi \pmod q$  with  $q \neq 1$, $L(s,\chi)\notin \mathbb{S}$.
	Recently, in \cite{kac2022,kac2024}, they classified the $L$-functions of degree $2$ and conductor $1$. They showed that every $F \in \mathbb{S}$ of degree $2$ and conductor $1$ is an automorphic $L$-function. The known examples of degree $2$ functions include products of two degree $1$ functions, such as $\zeta^2(s)$ and $L(s,\chi_1)L(s,\chi_2)$; normalized  $L$-functions associated with holomorphic newforms; the Dedekind zeta-function $\zeta_K(s)$ associated with a quadratic number field $K$; and the imaginary shifts of any of the above functions which are entire.  Normalized $L$-functions associated with non-holomorphic newforms are expected to lie in $\mathbb{S}$, but the Ramanujan hypothesis has not yet been verified, the best current upper bound for the Fourier coefficients is $n^{7/64+\epsilon}$ due to Kim and Sarnak \cite{kim2003refined}. Assuming the Ramanujan hypothesis holds, we believe these $L$-functions belong to $\mathbb{S}$. In these cases, infinitely many zeros are known to be on the critical line. However, for degree $2$, we do not yet know whether the known list contains all possible cases.
	Thus, the problem is still open in general. Under some suitable assumptions, Mukhopadhyay, Srinivas, and Rajkumar
	\cite{Mukhopadhyay2008} proved that degree-two functions in the Selberg class have infinitely many zeros on the critical line.  Their result applies,
	in particular, to the $L$-functions $L_f(s)$ attached to cuspidal modular
	forms $f$ of level $\mathcal{N}$ when $\mathcal{N}$ is not a perfect
	square, and also to the products $L(s,\chi_1)L(s,\chi_2)$ with primitive Dirichlet characters  $\chi_1 \pmod { q_1}$ and  $\chi_2 \pmod {q_2}$, 
	whenever $q_1q_2$ is not a perfect square.

	The functional equation \eqref{functional equation} can be written as 
	\begin{equation}\label{fun eq F=SF bar}
		F(s)=H(s)\overline{F}(1-\bar{s}),
	\end{equation}
	where
	\begin{equation}\label{Hs}
		H(s):=\omega Q^{1-2s}\prod_{i=1}^{m} \frac{\Gamma(\alpha_i (1-s)+\bar{\beta_i})}{\Gamma(\alpha_i s +\beta_i)}.
	\end{equation}
	This allows us to define the analogue of Hardy's function for $F \in \mathbb{S}$ as 
	\begin{equation}\label{deine hardy's ZF}
		Z_F(t):=F(1/2+it) H(1/2+it)^{-1/2}.
	\end{equation}
Using \eqref{fun eq F=SF bar}, we have
$\left|H(1/2+it)\right|=1$, and hence
$Z_F(t)=\overline{Z_F(t)}$.
Thus, $Z_F(t)$ is real valued for real $t$, and its zeros correspond exactly to the zeros of $F(1/2+it)$.

	Let us define
	\begin{align*}
		Z_F^+(t) & : = \max (0,Z_F(t)),\\
		Z_F^-(t) &: = \max (0,-Z_F(t)).
	\end{align*}
	%%%%%%%%%%%%%%  need more 
	The second author \cite{mahatab2019large} established $\Omega_{\pm}$ results for Hardy's function $Z(t)$ corresponding to the Riemann zeta function $\zeta(s)$ using the resonator constructed by Bondarenko and Seip \cite{Bondarenko2017}. Specifically, it was shown that,
	\begin{align*}
		\max_{T^{3/4}\leq t\leq T} \pm Z(t) &\gg \exp\left(C \sqrt{ \frac{{\log T \log \log \log T}}{{\log \log T}}  } \right),
	\end{align*}
	holds with the constant $C = 1/2 - \epsilon$ for every $\epsilon > 0$.
	
	It was also remarked in the paper that the interval $[T^{3/4},T]$ for $t$
	can be improved to $[T,2T]$ using Soundararajan's resonance method
	\cite{soundararajan2007extreme}, but at the cost of a weaker result.
	More precisely, the bound
	\begin{align}\label{Mahatab results}
		\max_{T\leq t\leq 2T}\pm Z(t)
		&\gg
		\exp\left(
		C\sqrt{\frac{\log T}{\log\log T}}
		\right),
	\end{align}
	holds with the constant $C=1/2-\epsilon$ for every $\epsilon>0$.
	In a related direction, Gonek and Ivi\'c~\cite{GonekIvic} and
	Mawia~\cite{Mawia} studied the interval of distribution of positive and negative
	values of Hardy's function. Further large-value results using the
	resonance method were obtained by Aistleitner and
	Pa\'nkowski~\cite{aistleitner2017large} for $L$-functions in the Selberg
	class, and by Peyrot~\cite{Peyrot2020} for Hecke--Maass $L$-functions
	on the critical line with prescribed argument.

	In this paper, we extend the above $\Omega_{\pm}$ result for Hardy's
	function to $L$-functions of degree $1$ and certain $L$-functions of
	degree $2$ using the resonance method. Here and throughout, the notation
	$Z_F^{\pm}(t)$ means that the corresponding estimate holds separately
	for both $Z_F^{+}(t)$ and $Z_F^{-}(t)$. We now state the results of this
	paper.
	
	\begin{thm}\label{thm:S1}
		Let $F(s)$ be an element of the Selberg class $\mathbb{S}$ of degree $1$.
		Then, for sufficiently large $T$,
		\[
		\max_{T\leq t\leq 2T} Z_F^{\pm}(t)
		\gg
		\exp\left(
		\left(\frac{1}{2}-\epsilon\right)
		\sqrt{\frac{\log T}{\log\log T}}
		\right),
		\]
		where $\epsilon>0$ is arbitrarily small.
	\end{thm}

	Before stating our second theorem, we introduce the following hypothesis
	for degree-two functions.
	
	\begin{hyp}\label{hyp}
		Let $F(s) \in \mathbb{S}_{\text{poly}}$  have degree $2$ and  $a_F(n)$ be the Dirichlet  coefficients of $F$. There exists $\gamma < 1$ and $\gamma_1 \in \mathbb{R}$ such  that
		\begin{equation*}
			\sum_{n \leq T} a_F(n) e^{2\pi i n x} \ll T^{\gamma} (\log T)^{\gamma_1},
		\end{equation*}
		uniformly for all large $T$ and $x \in \mathbb{R}$, with  implied constant depends  only on $F$. Here the values of $\gamma$ and $\gamma_1$ depend on $F(s)$.
	\end{hyp}

	\begin{thm}\label{thm general}
		Let $F (s) =\sum_{n=1}^{\infty}a_F(n)n^{-s}\in\mathbb{S}_{\mathrm{poly}}$ be a function of degree $2$
		satisfying Hypothesis~\ref{hyp}, with $ \Im \xi_F =0$, where $\xi_F $ is the $\xi$-invariant defined in \eqref{zai invariant}. Moreover, assume that $F(s)$ satisfies the following Selberg normality conjecture,
		\be\label{normality conjectures}
		\sum_{p \leq y} |a_F(p)|^2 = (\kappa + o(1)) \frac{y}{\log y},\qquad (\kappa >0).
		\ee
		Then for sufficiently large $T$, 
		\[
		\max_{T\leq t\leq 2T} Z_F^{\pm}(t)
		\gg
		\exp\left(C_F
		\sqrt{\frac{\log T}{\log\log T}}
		\right),
		\]
		where the constant $C_F$ depends on $F(s)$. More precisely, for arbitrarily small $\epsilon>0$, we have 
		\begin{equation}\label{value of CF}
			C_F= \sqrt{\kappa(1-\gamma)}-\epsilon.
		\end{equation}
	\end{thm}
	
\begin{rem}
	The condition $\Im \xi_F=0$ is imposed to avoid the additional factor
	$t^{-i\Im \xi_F}$ in \eqref{New form of Hs d}, whose presence makes the direct
	application of Lemma~\ref{lemma three deriv expon} to
	\eqref{main integral of I} more complicated. This condition does not restrict
	the result when $F(s)$ is regular at $s=1$. Indeed, by considering the shift
	$G(s)=F(s+i\tau)$ with $\tau=-{\Im \xi_F}/{d_F}$, we can achieve
	the condition $\Im\xi_G=0$. With a compatible choice of the square root,
	$
	Z_G(t)=Z_F(t+\tau).
	$
	Therefore, applying the same argument to $G$ on the translated interval
	$[T-\tau,2T-\tau]$ yields the conclusion of the above theorem for the original
	function $F$ on $[T,2T]$.
\end{rem}

	\begin{rem}
		Parseval's identity and \eqref{normality conjectures} give
		\[
		\int_0^1 \left|\sum_{n\leq Y} a_F(n)e^{2\pi i nx}\right|^2\,dx
		=
		\sum_{n\leq Y}|a_F(n)|^2
		\geq
		\sum_{p\leq Y}|a_F(p)|^2
		\sim
		\kappa\frac{Y}{\log Y}.
		\]
		On the other hand, Hypothesis~\ref{hyp} implies
		\[
		\int_0^1 \left|\sum_{n\leq Y} a_F(n)e^{2\pi i nx}\right|^2\,dx
		\ll
		Y^{2\gamma}(\log Y)^{2\gamma_1}.
		\]
		Comparing these two estimates, we must have $\gamma\geq 1/2$. Therefore, under Hypothesis~\ref{hyp}, it follows that
		$
		1/2\leq\gamma<1.
		$
	\end{rem}

	%\color{blue}
	As a special case of Theorem~\ref{thm general}, we get the following theorems and give a sketch of their proofs, as the arguments are similar.
	
	%\normalcolor
	
	\begin{thm}\label{thm:S2}
		Let $F(s)$ be one of the following functions:
		\begin{enumerate}
			\renewcommand{\labelenumi}{(\roman{enumi})}
			\item the $L$-function associated with a normalized primitive holomorphic Hecke cusp eigenform of even weight $k\geq 2$, level $\mathcal{N}$ and nebentypus $\chi$.
			\item  assuming the Ramanujan hypothesis,  the $L$-function associated with a normalized primitive Hecke-Maass cusp form of weight $0$, level  $\mathcal{N}$  and nebentypus $\chi$, with Laplace eigenvalue $1/4+r^2$ and $r\in \mathbb{R}$.
			
		\end{enumerate}
		Then for sufficiently large $T$,
		\[
		\max_{T\leq t\leq 2T} Z_F^{\pm}(t)
		\gg
		\exp\left(C_F
		\sqrt{\frac{\log T}{\log\log T}}
		\right),
		\]
		where the constant $C_F$ depends on $F(s)$, and can be computed using \eqref{value of CF}.
	\end{thm}

%	\begin{rem}
%		From Lemmas~\ref{lemma exponential sum} and \ref{normality cojecture for both Maass and modular}, we have $\gamma=1/2$ and $\kappa=1$. Hence, the constant $C_F$ in the above theorem is
%		$
%		{1}/{\sqrt{2}}-\epsilon,
%		$
%		where $\epsilon>0$ is an arbitrarily small number.
%	\end{rem}

	%%%%%%%%%%%%%

	The Ramanujan hypothesis in part~(ii) of Theorem~\ref{thm:S2}
	can be removed at the cost of a smaller constant. More precisely, we have the following theorem.
	
	\begin{thm}\label{thm:Maass-unconditional}
		Let $F(s)$ be the $L$-function associated with a normalized primitive
		Hecke--Maass cusp form of weight $0$, level $\mathcal{N}$ and
		nebentypus $\chi$, with Laplace eigenvalue $1/4+r^2$, where
		$r\in\mathbb{R}$. Then, unconditionally, for sufficiently large $T$,
		\[
		\max_{T\leq t\leq 2T} Z_F^{\pm}(t)
		\gg
		\exp\left(
		\left(\sqrt{\frac{25}{71}}-\epsilon\right)
		\sqrt{\frac{\log T}{\log\log T}}
		\right),
		\]
		where $\epsilon>0$ is arbitrarily small.
	\end{thm}

%	\begin{rem}
%		The proof of Theorem~\ref{thm:Maass-unconditional} follows by modifying
%		the proof of Theorem~\ref{thm general} using the bound
%		\[
%		a_F(n) \ll n^{\theta+\epsilon},
%		\qquad
%		\theta=\frac{7}{64},
%		\]
%		instead of the Ramanujan hypothesis. The required modifications are
%		given in Subsection~\ref{subsec:unconditional-Maass}.
%	\end{rem}
	%

	%%%%%%%%%%%%%%

	It is well known that the Dedekind zeta function
	$\zeta_K(s)=\sum_{n=1}^{\infty}a_K(n)n^{-s}$
	attached to a quadratic field $K$ has degree $2$. However, it does not satisfy Hypothesis~\ref{hyp} uniformly for all $x\in\mathbb{R}$, since for $x=0$,
	$$
	\sum_{n\leq T} a_K(n)\sim r_K T,
	$$
	where $r_K$ denotes the residue of $\zeta_K(s)$ at $s=1$. Therefore, the corresponding omega result for $\zeta_K(s)$ does not follow as a special case of Theorem~\ref{thm general}.
	
	We also consider the degree $2$ function $F(s)=L(s,\chi_1)L(s,\chi_2)$, where $\chi_1$ and $\chi_2$ are primitive Dirichlet characters modulo $q_1$ and $q_2$, respectively, with $q_1q_2$ not a perfect square. For this family also, the uniform exponential sum estimate required in our general argument is not directly available in the required form. Therefore, we treat these two families separately in the following theorem.
	
	\begin{thm}\label{thm for dedekind zeta function}
		Let $F$ be one of the following functions:
		\begin{itemize}
			\item[(i)] $F(s)=\zeta_K(s)$, where $K$ is any quadratic field;
			
			\item[(ii)] $F(s)=L(s,\chi_1)L(s,\chi_2)$, where $\chi_1$ and $\chi_2$ are primitive Dirichlet characters modulo $q_1$ and $q_2$, respectively, with $q_1q_2$ not a perfect square.
		\end{itemize}
		Then for sufficiently large $T$, 
		\[
		\max_{T\leq t\leq 2T} Z_F^{\pm}(t)
		\gg
		\exp\left(
		\left(\frac{1}{\sqrt2}-\epsilon\right)
		\sqrt{\frac{\log T}{\log\log T}}
		\right),
		\]
		where $\epsilon >0$ is an arbitrarily small number.
	\end{thm}

% \begin{rem}
% hh
% \end{rem}

\subsection{Strategy of the paper}

In this paper, we try to impliment the philosophy that if an $L$-function has infinitely many zeros on the critical line and produces large values there by means of a resonator, then, together with suitable cancellation in the corresponding signed resonated moment, one can obtain large positive and negative values of the associated Hardy function of comparable magnitude. This cancellation is essential, since large absolute values and infinitely many zeros on the critical line alone do not guarantee large values of both signs. Accordingly, we exclude cases such as $\zeta^2(s)$ and $L^2(s,\chi)$, for which the associated Hardy functions do not exhibit both positive and negative values.

To work on this idea, a natural approach is to use the second author's argument~\cite{mahatab2019large} for the Riemann zeta function, with some modifications. This idea works well for Dirichlet $L$-functions but fails for degree $2$ $L$-functions in the Selberg class due to the large error in the Voronoi summation formula. In particular, the proposition of~\cite{mahatab2019large} does not apply here directly. So, we modify it in our new Proposition~\ref{New prop R(t) phi} and show that the integral $\int_{T}^{2T} Z_F(t)\, |R(t)|^2 \, dt$ is small by making a shift in the contour. This idea is well known and has been used in \cite{chandrasekharan1968zeta,Meher2017pujahari,Mukhopadhyay2008} to show the existence of infinitely many zeros for degree $2$ $L$-functions. However, the computation in our case requires uniform exponential-sum estimates involving the Dirichlet coefficients $a_F(n)$ of the $L$-function (see \eqref{Final JF}). These estimates are readily available for $L$-functions associated with holomorphic cusp form and Maass cusp form. For the Dedekind zeta function of a quadratic field and for $L(s,\chi_1)L(s,\chi_2)$, we establish the required exponential-sum estimate in Section~\ref{proof for dedekind}.

%\color{blue}
The method is not restricted to functions in the Selberg class and can also be applied to certain $L$-functions outside this class, provided that the corresponding Hardy function and the required exponential-sum estimates are available. We discuss such extensions in Section~\ref{section beyond selberg}.

%\normalcolor

	\begin{rem}
		For $L$-functions of degree $3$ or higher,
		a major difficulty in extending our method
		is obtaining a sufficiently strong upper bound for
		$
		\int_T^{2T} Z_F(t)|R(t)|^2\,dt.
		$
		In particular, the techniques used in this paper
		do not appear to provide the required cancellation
		in this integral for higher-degree $L$-functions.
		Thus, extending our results to degrees greater
		than $2$ remains a challenging problem.
	\end{rem}

	\section{Preliminaries}\label{section preli}
	\subsection{ Asymptotic and order estimates}
	
	We know from the  Stirling formula for the $\Gamma$-function that in any fixed vertical strip $-\infty < a \leq \sigma \leq b < \infty$, we have as $t \to \infty$:
	\[
	\Gamma(\sigma + it) = \sqrt{2\pi} \, t^{\sigma - \frac{1}{2} + it} e^{-\frac{\pi}{2} t} e^{\frac{\pi}{2} i (\sigma - \frac{1}{2}) - it} \left(1 + O\!\left(\frac{1}{t}\right)\right).
	\]
	To apply this to the ratio of Gamma factors in \eqref{Hs}, let $s = \sigma + it$. For each $j = 1, \dots, m$, the ratio expands as:
	\begin{equation}\label{ratio of gamma}
		\begin{aligned}
			\frac{\Gamma(\alpha_j (1-s)+\bar{\beta_j})}{\Gamma(\alpha_j s +\beta_j)} &= (\alpha_j t + \Im\beta_j)^{\alpha_j(1-2\sigma) - 2i(\alpha_j t + \Im\beta_j)} e^{-\frac{\pi}{2}i(\alpha_j + 2\Re\beta_j - 1)} e^{2i(\alpha_j t + \Im\beta_j)} \left(1 + O\!\left(\frac{1}{t}\right)\right) \\
			&= (\alpha_j t)^{\alpha_j(1-2\sigma) - 2i(\alpha_j t + \Im\beta_j)} e^{-2i\Im\beta_j} e^{-\frac{\pi}{2}i(\alpha_j + 2\Re\beta_j - 1)} e^{2i\alpha_j t} e^{2i\Im\beta_j} \left(1 + O\!\left(\frac{1}{t}\right)\right) \\
			&= \alpha_j^{\alpha_j(1-2s) - 2i\Im\beta_j} t^{\alpha_j(1-2s)} t^{-2i\Im\beta_j} e^{-\frac{\pi}{2}i(\alpha_j + 2\Re\beta_j - 1)} e^{2i\alpha_j t} \left(1 + O\!\left(\frac{1}{t}\right)\right), 
		\end{aligned}
	\end{equation}
	
	after simplification and by using 
	\[
	\left(1 + \frac{\Im\beta_j}{\alpha_j t}\right)^{-2i\alpha_j t} = e^{-2i\Im\beta_j} \left(1 + O\!\left(\frac{1}{t}\right)\right). 
	\]
	Applying this to the full product for $H(s)$, we have as $t \to \infty$:
	\begin{align*}
		H(s) = & \, \omega Q^{1-2s} t^{\sum_{j=1}^{m}\alpha_j (1-2s)} t^{-2i \sum_{j=1}^{m} \Im{\beta_j}} e^{-\frac{\pi}{2}i\sum_{j=1}^{m} ( \alpha_j+ 2 \Re\beta_j-1)} \\
		& \times e^{2it\sum_{j=1}^{m} \alpha_j} \prod_{j=1}^{m} \alpha_j^{\alpha_j(1-2s)-2i\Im\beta_j} \left(1 + O\!\left(\frac{1}{t}\right)\right).
	\end{align*}
	Recall the degree $d_F:=2 \sum_{j=1}^{m}\alpha_j$, the $\xi$-invariant $\xi_F := 2 \sum_{j=1}^{m}  \left ( \beta_j-\frac{1}{2} \right)$, and the conductor $q_F:= (2 \pi )^{d_F} Q^2 \prod_{j=1}^{m} \alpha_{j}^{2 \alpha_j}$. We simplify the above with $d_F$ and $\xi_F$, and rewrite $H(s)$ as
	\begin{align}\label{New form of Hs  d}
		H(s) = \omega_o \left( Q_o t^{d_F/2}\right)^{1-2s} t^{-i \Im{\xi_F}} e^{i d_F t} \left(1 + O\!\left(\frac{1}{t}\right)\right),
	\end{align}
	where $\omega_o = \omega e^{-\frac{\pi i}{2}(d_F/2+\Re\xi_F)} \prod_{j=1}^{m} \alpha_j^{-2i\Im\beta_j}$ with $|\omega_o|=1$ and $Q_o= Q \prod_{j=1}^{m} \alpha_j^{\alpha_j}>0$. Also note that $Q_o=\sqrt{q_F}(2 \pi)^{-d_F/2}$.
	
	In this article, we restrict ourselves to $L$-functions in the Selberg class of degree $2$ with  $ \Im \xi_F =0$. More specifically, we investigate  Hardy's $Z$-function associated to the Dedekind zeta function over a quadratic field, the $L$-function associated with normalized Hecke primitive eigen cusp forms of weight $k$ and level  $\mathcal{N}$, and the $L$-function associated with a Maass cusp form of weight $0$ and level  $\mathcal{N}$ corresponding to an eigenvalue $1/4+r^2$ of the hyperbolic Laplacian. Thus, for $L$-functions of degree $2$ satisfying $\Im \xi_F=0$, we have 
	\begin{align*}
		H(s)=\omega_1 (Q_ot)^{1-2\sigma -2i t}  e^{2it}\left(1 + O\!\left(\frac{1}{t}\right)\right),
	\end{align*}
	where $\omega_1=\omega e^{-\frac{\pi i}{2}(1+\Re\xi_F)} \prod_{i=1}^{m} \alpha_i ^{-2i\Im\beta_i}$, with $|\omega_1|=1$ and $Q_o=\sqrt{q_F}/{2\pi}$.
	
	Therefore,
	\begin{align}\label{Hs -1/2}
		H^{-1/2}(s) =\omega_2 (Q_ot)^{\sigma-1/2+it}e^{-it} \left(1 + O\!\left(\frac{1}{t}\right)\right) ,
	\end{align}
	where $\omega_2=(\omega_1)^{-1/2}$  with $|\omega_2|=1.$
	
	Now to find the order of  Hardy's function and order of its analytic extension in a neighbourhood, we need the following convexity bound for $F(s)$.
	
	\begin{lem}\label{estimate of Fs interval}
		Let $F(s) \in \mathbb{S}_2$. Then for any  constant  $\delta > 0$, the $L$-function  $F(s)$ satisfies
		\[
		F(\sigma + it) =O\left ( t ^{1 - \sigma + \delta}\right) \quad \text{for} \quad  - \delta \leq \sigma \leq  1 + \delta.
		\]
	\end{lem}
	\begin{proof}
		Using Stirling's formula for $\Gamma(s)$, from the functional equation \eqref{fun eq F=SF bar}, we have for $\delta >0$,
		\[
		F(-\delta+it) \ll t^{1+2\delta}
		.    \]
		Since the Dirichlet series
		$
		F(s)=\sum_{n=1}^{\infty}\frac{a_F(n)}{n^s}
		$
		converges absolutely in the half-plane $\Re(s)>1$, it follows that, for any
		$\delta>0$,
		\[
		F(1+\delta+it)\ll 1.
		\]
		Then applying the appropriate Phragm\'{e}n--Lindel\"{o}f \cite[Theorem 5.53]{iwaniec2004analytic} argument, we get the required result.
	\end{proof}

	Using  \eqref{Hs -1/2} together with Lemma \ref{estimate of Fs interval}, we obtain that for constant  $\delta >0$  and $  - \delta \leq \sigma \leq 1 + \delta,$ 
	\begin{align}\label{esti H -1/2 and F}
		F(\sigma+it) H^{-\frac{1}{2}}(\sigma+it) = O\left ( t^{\frac{1}{2}+\delta}\right ), \quad \text{as} \quad t \to \infty.
	\end{align}
	
	We now turn to certain exponential sums and integral estimates, which are the main tools we need. These estimates will be used both in preparing our auxiliary results and in the final proof of our main theorems.
	\subsection{Exponential sums and integral estimates}
	We now state the following results on exponential integrals, which will be useful in proving our main results.
	
	\begin{lem}\cite[Lemma 4.3]{titchmarsh1986theory}.\label{lemma tith 1st der}
		Let $f(x)$  and   $g(x)$  be  real functions such that  $f'(x)/g(x)$ is monotonic and $|f'(x)/g(x)| \geq m >0$, throughout the interval $[a,b]$. Then 
		\begin{align*}
			\left |  \int_a^b g(x)\,e^{i f(x)} \, dx \right | \leq \frac{4}{m}.
		\end{align*}
	\end{lem}

	\begin{lem}\cite[Lemma 2.3]{ivic2012theory}.\label{lem:sec der test}
		Let $f(x)$ be a real twice differentiable function  such that $|f''(x)| \ge \lambda_2 >0$ for $x \in [a,b]$. Let $g(x)$ be a positive monotonic function such that $|g(x) |\le G$ on $[a,b]$. Then 
		\[
		\int_a^b g(x) e^{if(x)} \, dx \le \frac{8G}{\sqrt{\lambda_2}}.
		\]
	\end{lem}

	\begin{lem}\cite[Lemma 4.3]{rajkumar2012}.\label{lemma three deriv expon}
		Let $f(x)$  be a real function in $[a,b]$,  which is three times continuously differentiable, such that $f'(x)$ is monotonic, $f'(c)=0$ for $a<c<b$, $|f''(x)| \asymp\lambda_2$ and $0<|f'''(x)|\ll \lambda_3.$ Let $g(x)$ be a real function which is two times continuously differentiable, with $|g(x)|\leq G,|g'(x)|\leq G_1$ for $x \in [a,b]$. Suppose that the functions
		\begin{equation*}
			f'^{6}(x)-8f''(c)f''^2(x)(f(x)-f(c))^3 \text { and }  g''(x)
		\end{equation*}
		have only a bounded number of zeros in $[a,b].$ Then,  
		\begin{align*}
			\ \int_a^b g(x)\,e^{i f(x)} \, dx   &   = g(c) (2\pi)^{1/2} \frac{e^{\pm i\pi/4+i f(c)}}{|f''(c)|^{1/2}}+O\left(\frac{G_1}{\lambda_2}\right) +  O\left(\frac{G(b-a)\lambda_3^2}{\lambda_2^3}\right) \\
			&+O\left(\frac{G\lambda_3}{\lambda_2^2}\right)
			+  O\left(  G \min \left( \frac{1}{|f'(a)|}+\frac{1}{|f'(b)|},\lambda_2 ^{-1/2}\right)  \right).
		\end{align*}
		Here, the sign in $e^{\pm i\pi/4}$ is chosen as positive or negative according as $f''$ is positive or negative.
	\end{lem}

	We next state a bound for exponential sums involving the coefficients of integral weight cusp forms or Maass cusp forms.
	
	\begin{lem}\label{lemma exponential sum}
		Let $f$ be either a normalized primitive holomorphic Hecke cusp eigenform of even weight $k\geq 2$, or a normalized primitive Hecke--Maass cusp form of weight $0$ with Laplace eigenvalue $1/4+r^2$, where $r\in\mathbb{R}$. In each case, let $f$ have level $\mathcal{N}$ and nebentypus $\chi$. If $\lambda_f(n)$ denotes the normalized Fourier coefficient of $f$, then
		\begin{equation*}
			\sum_{n\leq T}\lambda_f(n)e^{2\pi i n x}
			\ll T^{1/2}\log(2T),
		\end{equation*}
		for $T\geq 1$ and $x\in\mathbb{R}$, where the implied constant depends only on $f$.
	\end{lem}
	
	\begin{proof}
		The result follows from  \cite[Lemma~5.3]{iwaniec1997topics} and \cite[Theorem~8.1]{iwaniec1995introduction}.	
	\end{proof}

	We now present the cases needed in this paper for which the Selberg normality conjecture \eqref{normality conjectures} is known to hold.
	
	\begin{lem}\cite[Theorem 2.2]{humphries2022jesse}\label{normality cojecture for both Maass and modular}
		Let $f$ be either a normalized primitive holomorphic Hecke cusp eigenform of even weight $k\geq 2$, or a normalized primitive Hecke--Maass cusp form of weight $0$ with Laplace eigenvalue $1/4+r^2$, where $r\in\mathbb{R}$. In each case, let $f$ have level $\mathcal{N}$ and nebentypus $\chi$. If $\lambda_f(n)$ denotes the normalized Fourier coefficient of $f$, then
		\[
		\sum_{p\leq x} |\lambda_f(p)|^2=(1+o(1))\frac{x}{\log x}.
		\]
	\end{lem}

	%\color{blue}
	\begin{lem}\label{lemma 4th moment}
		Let $f$ and $\lambda_f(n)$ be as in Lemma \ref{normality cojecture for both Maass and modular}. Then we have 
		\[
		\sum_{p\leq x}|\lambda_f(p)|^4 \ll \frac{x}{\log x},
		\]
		where the implied constant depends on $f$.
	\end{lem}

\begin{proof}
	For holomorphic Hecke cusp forms, we have
	$|\lambda_f(p)|\leq 2$ for $p\nmid N$.
	Hence, the result follows in this case.
	
	Now, for Hecke--Maass cusp forms, $f$ is either
	dihedral or non-dihedral. If $f$ is dihedral,
	then $|\lambda_f(p)|\leq 2$ for $p\nmid N$,
	and the result follows.
	
	Now assume that $f$ is non-dihedral.
	Let $L(s,\operatorname{Ad}(f))$ denote
	the adjoint square $L$-function of $f$ with
	the Euler factors at primes dividing $N$
	omitted, defined for $\Re(s)>1$ by
	\[
	L(s,\operatorname{Ad}(f))
	:=\sum_{n\geq 1}\frac{A_f(n)}{n^s},
	\]
	where
	\[
	A_f(n)=
	\begin{cases}
		\displaystyle
		\sum_{k^2\mid n}
		\lambda_f\left(\frac{n^2}{k^4}\right)
		\overline{\chi\left(\frac{n}{k^2}\right)},
		& \text{if } (n,N)=1,\\[8pt]
		0, & \text{otherwise}.
	\end{cases}
	\]
	For $p\nmid N$, we have
	\[
	A_f(p)=\lambda_f(p^2)\overline{\chi(p)}
	=|\lambda_f(p)|^2-1,
	\]
	where the last equality follows from the
	Hecke relations  $\lambda_f(p^2) \overline{\chi(p)} = |\lambda_f(p)|^2-1$.
	Since $f$ is non-dihedral, its adjoint lift
	is a self-contragredient cuspidal automorphic
	representation of $\mathrm{GL}_3$.
	Hence, by \cite[Lemma~3.2]{liu2015}
	and partial summation, we have
	\[
	\sum_{p\leq x}|A_f(p)|^2
	\ll \frac{x}{\log x}.
	\]
	Therefore,
	\begin{align*}
		\sum_{\substack{p\leq x\\p\nmid N}}
		|\lambda_f(p)|^4
		=\sum_{\substack{p\leq x\\p\nmid N}}
		(1+A_f(p))^2
		\ll \sum_{p\leq x}(1+|A_f(p)|^2)
		\ll \frac{x}{\log x}.
	\end{align*}
	Since the primes dividing $N$ contribute
	$O(1)$, the desired result follows.
\end{proof}
%\normalcolor

	\begin{lem}{\cite[Lemma 5.2]{milinovich2014moments}.}\label{normality conjecture dedekind}
		Let $\zeta_K(s)=\sum_{n=1}^{\infty}a_K(n)n^{-s} $ be the Dedekind zeta function associated to  any  quadratic field $K$. Then 
		\[
		\sum_{p\leq x}a^2_K(p) \ =(2+o(1)) \frac{x}{\log x}.
		\]
	\end{lem}

	\begin{lem}\label{lemma normality conjecture for prod}
		Let $F(s)=L(s,\chi_1)L(s,\chi_2),$
		where $\chi_1$ and $\chi_2$ are primitive Dirichlet characters modulo $q_1$ and $q_2$, respectively, with $q_1q_2$ not a perfect square. Then
		\[
		\sum_{p\le x}|a_F(p)|^2
		=(2+o(1))\frac{x}{\log x}.
		\]
	\end{lem}
	
	\begin{proof}
		Since
		$
		a_F(p)=\chi_1(p)+\chi_2(p),
		$
		we have
		\[
		|a_F(p)|^2
		=|\chi_1(p)|^2+|\chi_2(p)|^2
		+\chi_1(p)\overline{\chi_2(p)}
		+\overline{\chi_1(p)}\chi_2(p).
		\]
		View both characters modulo $\operatorname{lcm}(q_1,q_2)$. Since $q_1q_2$ is not a perfect square, we have $q_1\neq q_2$; hence, by primitivity, $\chi_1\neq\chi_2$, and therefore $\chi_1\overline{\chi_2}$ is nonprincipal. By the prime number theorem for Dirichlet characters, the two cross terms contribute $o(x/\log x)$, while each diagonal term contributes $(1+o(1))x/\log x$. This proves the result.
	\end{proof}

	\subsection{Resonance method}
	
	The resonator $R(t)$ constructed by Soundararajan in \cite{soundararajan2007extreme} has  the form of a Dirichlet polynomial
	\be\label{resonator}
	R(t) = \sum_{m \leq N} r(m) m^{-it} \quad \text{with} \quad  N \leq T^\alpha, \alpha=1-\epsilon, \, \epsilon >0.
	\ee
	Choose a smooth function $\Phi$ supported on $[1,2]$ with $0\leq \Phi(t)\leq 1$, and assume that $\Phi(t)=1$ for $t\in(5/4,7/4)$. Define its Fourier transform by
	$$
	\widehat{\Phi}(y)=\int_{\mathbb{R}}\Phi(t)e^{-ity}\,dt.
	$$
	Then, for any fixed integer $A\geq 1$, it satisfies the decay estimate $\widehat{\Phi}(y)\ll |y|^{-A}$.

	%%%%%%%%%%%%%%%%%%
%	\normalcolor
%	\color{blue}
	
	For a small $0<\eta<1$. Put
	\be\label{L^2}
	L^2=\frac{(1-\eta)\log N\log\log N}{\kappa},
	\qquad
	M=\exp\big((\log L)^2\big),
	\ee
	and define the nonnegative multiplicative function $f$, supported on
	squarefree integers, by
	\begin{equation}\label{f(p)}
		f(p)=
		\begin{cases}
			\dfrac{L}{\sqrt{p}\log p}, & L^2\leq p\leq M,\\[2mm]
			0, & \text{otherwise},
		\end{cases}
	\end{equation}
	with $f(p^j)=0$ for $j\geq 2.$ Set $r(n)=a_F(n)f(n).$

	\begin{lem}\label{lemma >=exp()}
		Let $F (s) =\sum_{n=1}^{\infty}a_F(n)n^{-s}$, where $a_F(n)$ is multiplicative and $a_F(1)=1$. Suppose that the coefficients $a_F(n)$ satisfy \eqref{normality conjectures} and 
		\be\label{4th moment}
		\sum_{ p \leq x } |a_F(p)|^4 \ll \frac{x}{\log x}.
		\ee
		 Then, 
		for the preceding resonator,
		\[
		\frac{1}{\sum_{n\leq N} |r(n)|^2}
		\sum_{nk\leq N}
		\frac{a_F(k)r(n)\overline{r(nk)}}{\sqrt{k}}
		\geq
		\exp\left(
		\left(\sqrt{\kappa(1-\eta)}+o(1)\right)
		\sqrt{\frac{\log N}{\log\log N}}
		\right).
		\]
		In particular, this bound holds  for $F \in \mathbb{S}_{\text{poly}}$ satisfies \eqref{normality conjectures}.

	\end{lem}
	
	\begin{proof}
		Since the nonnegative multiplicative function $f$ is supported on squarefree integers and $a_F(n)$ is multiplicative, we have
		\[
		\sum_{nk\leq N}
		\frac{a_F(k)r(n)\overline{r(nk)}}{\sqrt{k}}
		=
		\sum_{\substack{nk\leq N\\(n,k)=1}}
		\frac{|a_F(k)|^2|a_F(n)|^2f(k)f(n)^2}{\sqrt{k}}
		=: S_N,
		\]
		which is a non-negative real sum. Removing the restriction $nk \leq N$, let
	\be\label{P and D}
	\begin{aligned}
		P
		&=
		\prod_p
		\left(
		1+|a_F(p)|^2f(p)^2
		+\frac{|a_F(p)|^2f(p)}{\sqrt{p}}
		\right)
		=
		\prod_p (1+u_p+v_p),\\
		D
		&=
		\prod_p
		\left(
		1+|a_F(p)|^2f(p)^2
		\right)
		=
		\prod_p (1+u_p),
	\end{aligned}
	\ee
	where
	\[
	u_p=|a_F(p)|^2f(p)^2,
	\qquad
	v_p=\frac{|a_F(p)|^2f(p)}{\sqrt{p}}.
	\]
		Note that $ \sum_{n\leq N} |r(n)|^2 \leq D$ and set $\rho=(\log L)^{-3} $.
		Rankin's trick gives
		\[
		0\leq P-S_N
		\leq
		N^{-\rho}
		\prod_p
		\left(
		1+ (u_p+v_p) p^\rho
		\right).
		\]
		Hence, we obtain
		\[
		\frac{P-S_N}{P}
		\leq
		\exp\left(
		-\rho\log N
		+
		\sum_p 
		(u_p+v_p)(p^\rho-1)
		\right).
		\]
		Since
		$
		\rho\log M=(\log L)^{-1}=o(1),
		$
		we have
		$
		p^\rho-1=(1+o(1))\, \rho\log p.$
		Partial summation gives
		\[
		\sum_{L^2\leq p\leq M}
		\frac{|a_F(p)|^2}{p\log p}
		=
		(\kappa+o(1))
		\left(
		\frac{1}{\log(L^2)}-\frac{1}{\log M}
		\right)
		=
		\frac{\kappa+o(1)}{2\log(L)}.
		\]
		Consequently,
		\[
		\begin{aligned}
			\sum_p
			|a_F(p)|^2
			\left(
			f(p)^2+\frac{f(p)}{\sqrt{p}}
			\right)\log p
			=
			(1-\eta+o(1))\log N.
		\end{aligned}
		\]
		Thus
		\[
		\frac{P-S_N}{P}
		\leq
		\exp\left(
		-(\eta+o(1))\rho\log N
		\right)
		=
		o(1).
		\]
		Hence
		$
		S_N\geq (1-o(1))P.$
      Now
	\[
	\begin{aligned}
		\log\frac{P}{D}
		&=
		\sum_p
		\log\left(
		1+
		\frac{v_p}{1+u_p}
		\right)
		=
		\sum_p
		\left(
		\frac{v_p}{1+u_p}
		+
		O\left(
		\left(
		\frac{v_p}{1+u_p}
		\right)^2
		\right)
		\right).
	\end{aligned}
	\]
Since 
   \[
  \left(  \frac{v_p}{1+u_p}  \right)^2 \ll v_p^2  \qquad \text{and} \qquad \frac{v_p}{1+u_p} = v_p + O(u_p v_p),
   \]
we get 
\begin{align*}
	\log\frac{P}{D}	&=  \sum_{p}v_p +O\left( \sum_p( u_pv_p +v_p^2) \right).
\end{align*}
Using \eqref{f(p)}, we have obtained 
\begin{align}\label{log P/D 4th mome}
		\log\frac{P}{D}	&= L\sum_{L^2 \leq p\leq M} \frac{|a_F(p)|^2}{p \log p} + O \left ( L^3  \sum_{L^2 \leq p\leq M} \frac{|a_F(p)|^4}{p^2 (\log p)^3} + L^2 \sum_{L^2 \leq p\leq M} \frac{|a_F(p)|^4}{p^2 (\log p)^2}    \right).
\end{align}
Using \eqref{normality conjectures}, \eqref{L^2}, and
partial summation, the main term of the above expression is
$$(\kappa+o(1))\frac{L}{2\log L}.$$
For the error term, we use the bound \eqref{4th moment}. Indeed, partial summation on  \eqref{4th moment} gives 
\[
\sum_{y\leq p} \frac{|a_F|^4}{p^2 (\log p)^j} \ll \frac{1}{y ( \log y) ^{j+1}} \qquad (j=1,2),
\]
so that the error term in \eqref{log P/D 4th mome} is \[
\ll \frac{L}{(\log L)^4}+ \frac{1}{(\log L)^3} =o\left( \frac{L}{\log L}\right)
\]
Hence
\begin{align*}
\log\frac{P}{D}	&=  (\kappa+o(1))\frac{L}{2\log L}\\
&= \left(\sqrt{\kappa(1-\eta)}+o(1)\right)
\sqrt{\frac{\log N}{\log\log N}},
\end{align*}
by using \eqref{L^2}. This completes the proof.
	\end{proof}

%	\normalcolor
	%%%%%%%%%%%%%%%

	The proof of Theorem \ref{thm:S1} is based on the following fundamental observation, which is analogous to the Proposition following Theorem 1 in \cite{mahatab2019large}. We do not include the  proof here since it  follows in exactly the same way as done  in \cite{mahatab2019large}. 
	
	\begin{prop}\label{mahatap prop}
		Let $F$ be an element of the Selberg class  $\mathbb{S}$ and $Z_F(t)$ denote  Hardy's $Z$-function corresponding to $F$. 
		Suppose there exists a non-negative function $K(t)$ and a function $A(T)$ such that 
		\begin{eqnarray*}
			A(T) \int_{T}^{2T} K(t) dt &\ll&  \int_{T}^{2T} F(1/2+it) K(t) dt, 
			\\
			\text{ and }  \int_{T}^{2T} Z_F(t) K(t) dt &=& o\left( A(T)  \int_{T}^{2T} K(t) dt \right), 
		\end{eqnarray*}
		then 	
		\[
		\max_{T\leq t\leq 2T} Z^+_F(t), \max_{T\leq t\leq 2T} Z^-_F(t) \gg A(T). 
		\]
	\end{prop}

	The proof of Theorem \ref{thm:S2} is  based on the following proposition, which is similar to  the one given above.
	
	\begin{prop}\label{New prop R(t) phi}
		Let $F$ be an element of the Selberg class  $\mathbb{S}$ and $Z_F(t)$ denote Hardy's $Z$-function corresponding to $F$. Suppose there is a function $A(T)$ such that
		\begin{eqnarray*}
			A(T) \int_{T^{}}^{2T}|R(t)|^2 \Phi(t/T) \, dt &\ll&  \int_{T^{}}^{2T} F(1/2+it)|R(t)|^2 \Phi(t/T) \, dt,
			\\
			\text{ and }  \int_{T^{}}^{2T} Z_F(t) |R(t)|^2 \, dt  &=&  o\!\left( A(T) \int_{T^{}}^{2T} |R(t)|^2 \, dt \right), 
		\end{eqnarray*}
		then 	
		\[
		\max_{T\leq t\leq 2T} Z^+_F(t), \max_{T\leq t\leq 2T} Z^-_F(t) \gg A(T). 
		\]
		where  $R(t)$, $\Phi(t)$  are as  defined  above Proposition \ref{mahatap prop} and \(N=\lfloor T^{\alpha}\rfloor\) with $0<\alpha<1$.
	\end{prop}

	\begin{proof}
		To prove the claim, observe that there exists a constant $C_1 > 0$ for which
		\begin{align}
			C_1 A(T) \int_T^{2T} |R(t)|^2 \Phi(t/T) \, dt \leq  \int_T^{2T}|Z_F(t)| |R(t)|^2 \Phi(t/T) \, dt \leq \int_T^{2T}|Z_F(t)| |R(t)|^2 \, dt . \label{C1 AT}
		\end{align}
		Since $R(t)=\sum_{n\leq N}r(n)n^{-it}$,
		by the mean-value theorem for Dirichlet polynomials
		(see \cite[Theorem~6.1]{Montgomery}), we have
		\begin{equation}\label{int Rt}
			\begin{aligned}
				\int_T^{2T}|R(t)|^2\,dt
				&=\left(T+O(N)\right)
				\sum_{n\leq N}|r(n)|^2\\
				&=T\left(1+O(T^{\alpha-1})\right)
				\sum_{n\leq N}|r(n)|^2\\
				&\ll T\sum_{n\leq N}|r(n)|^2,
			\end{aligned}
		\end{equation}
		where we have used $N\leq T^\alpha$ with
		$0<\alpha<1$. From  \cite[Eq.~(2)]{soundararajan2007extreme}, we have 
		\begin{align*}
			\frac{\int_{T}^{2T} |R(t)|^2 \, dt }{\int_{T^{}}^{2T}|R(t)|^2 \Phi(t/T) \, dt }  \ll \frac{ T \sum_{n \leq N} |r(n)|^2 }{T\widehat{\Phi}(0)(1+O(1/T)) \sum_{n\leq N} |r(n)|^2 } \ll 1.
		\end{align*}
		Hence, there exists a constant $C_2 >0$ such that $  \int_{T^{}}^{2T}|R(t)|^2 \Phi(t/T) \, dt  \geq C_2  \int_{T}^{2T} |R(t)|^2 \, dt  $ .
		
		Assume for contradiction that the claim is false and that
		\[
		\max_{T \leq t \leq 2T} Z_F^-(t) \leq \frac{C_1 C_2 A(T)}{3}.
		\]
		Then
		\begin{align*}
			\int_{T^{}}^{2T} Z_F(t) |R(t)|^2 \, dt  & = \int_{T^{}}^{2T} |Z_F(t)| |R(t)|^2 \, dt -\int_{T^{}}^{2T} 2Z^-_F(t) |R(t)|^2 \, dt  \\
			& \geq     C_1 A(T) \int_T^{2T} |R(t)|^2 \Phi(t/T) \, dt- \int_{T^{}}^{2T} 2Z^-_F(t) |R(t)|^2 \, dt \\
			& \geq \frac{C_1 C_2 A(T)}{3} \int_T^{2T} |R(t)|^2 \, dt,
		\end{align*}
		which contradicts the fact that  $\int_{T^{}}^{2T} Z_F(t) |R(t)|^2 \, dt$ is  $ o\!\left( A(T) \int_{T^{}}^{2T} |R(t)|^2 \, dt \right) $. This proves our claim for $Z^-_F(t)$. Similarly  we can argue for $ Z^+_F(t).$
	\end{proof}

	\section{Proof of Theorem \ref{thm:S1} }\label{section proof of thm 1}

	Let $F$ be an element of the Selberg class of degree $1$. We know that the  class $\mathbb{S}_1$ consists of the Riemann zeta function $\zeta(s)$ and the shifted Dirichlet $L$-functions
	$
	L_\theta(s,\chi):=L(s+i\theta,\chi).
	$
	Since the desired result is already known for the Riemann zeta function by \eqref{Mahatab results}, it remains to consider the latter case in the proof of Theorem \ref{thm:S1}.

	Let $L(s,\chi)$ denote the Dirichlet $L$-function associated with a primitive Dirichlet character $\chi$ of conductor $q$.
	From the functional equation for $L(s, \chi)$, we have  
	\begin{align}\label{fneq4}
		L(s+i\theta , \chi) = \psi(s+i\theta , \chi) L(1-s-i\theta, \overline{\chi}), 
	\end{align}
	with  
	\begin{align*}
		\psi(s, \chi) = \mathfrak{m}(\chi) \left(\frac{\pi}{q}\right)^{s-1/2} \frac{\Gamma\big( (1-s+\mathfrak{a}) /2 \big)}
		{\Gamma\big( (s+\mathfrak{a}) /2 \big)}, 
	\end{align*}
	where $\mathfrak{a} = (1-\chi(-1))/2$ measures the parity of the character $\chi$, $\mathfrak{m}(\chi) = \tau(\chi)/ (i^{\mathfrak{a}}\sqrt{q} )$ and $\tau(\chi) = \sum_{a \pmod q} \chi(a)e(a/q)$ is the Gauss sum; $\mathfrak{m}(\chi)$ is the  root number. 
	Comparing \eqref{fneq4} with \eqref{fun eq F=SF bar}, we see that Hardy's $Z$-function corresponding to $L(s+i\theta,\chi)$ is  
	\begin{align}
		Z_\theta(t,\chi) &:= L(1/2+it+i\theta, \chi) \psi(1/2+it+i\theta, \chi)^{-1/2},  
	\end{align}
	
	that is, a translate of Hardy's function $Z(t,\chi)$ corresponding to $L(s,\chi)$. Thus for our purpose, it suffices to work with $Z(t,\chi)$. 
	The following lemma gives the approximate functional equation for $Z(t,\chi)$.

	\begin{lem} \label{lem:AFE Z chi}
		For fixed $\chi$ and $t \to \infty$, we have 
		\begin{equation}
			Z(t,\chi)  =  2\Re\left\lbrace  c_\chi  \sum_{n \le \sqrt{ q |t|/2\pi  }} \frac{\chi(n)}{\sqrt{n}}  \exp\left( i t\left( \frac{1}{2}\log\left(  \frac{qt}{2\pi e} \right) - \log n \right)\right) \right\rbrace +O(|t|^{-1/4}), 
		\end{equation}
		where $c_\chi= 
		\mathfrak{m}(\chi)^{-1/2} e^{(\pi i \mathfrak{a}/ 4) - (\pi i/8)  }$. 
	\end{lem}
	%  \begin{proof}
		%  This follows from Lemma 2 of \cite{mawia2017distribution}. 
		%\end{proof} 

		%\normalcolor
		
		%\textbf{
			%	------- the above Lemma 3.1 we can proof by using Chandrasekharan and  Narasimhan  (1963) paper \cite{chandra1963} ------
			%}
		
		%\color{blue}
		\begin{proof}
			From  \cite[Eq.~(44)]{chandra1963} with $x=y=\sqrt{q|t|/{2\pi}}$, we have 
			\begin{equation*}
				L\left(\frac12+it,\chi\right)=\sum_{n \leq \sqrt { q|t|/{2 \pi}}} \frac{\chi(n )}{n^{\frac12+it}}+ \psi\left(\frac12+it, \chi\right) \sum_{n \leq \sqrt { q|t|/{2 \pi}}} \frac{\overline{\chi}(n)}{n^{\frac12-it}} +O(|t|^{-1/4}).
			\end{equation*} 
			Therefore,
			\begin{align*}
				L\left(\frac12+it,\chi\right)	\psi\left(\frac12+it, \chi\right)^{-1/2}  & = \psi\left(\frac12+it, \chi\right)^{-1/2}  \sum_{n \leq \sqrt { q|t|/{2 \pi}}} \frac{\chi(n )}{n^{\frac12+it}}\\
				&+\psi\left(\frac12+it, \chi\right)^{1/2} \sum_{n \leq \sqrt { q|t|/{2 \pi}}} \frac{\overline{\chi}(n)}{n^{\frac12-it}} +O(|t|^{-1/4}).
			\end{align*}
			Using \eqref{ratio of gamma} for the ratio of Gamma factors in $\psi\left(\frac12+it,\chi\right)$, the above equation completes the proof.
		\end{proof}
		\normalcolor

		Let us consider 
		\begin{equation}\label{I1}
			I_1 = I_1(\chi, R,T) := \int_\mathbb R Z(t,\chi) |R(t)|^2 \,\Phi(t/T)\, dt, 
		\end{equation}
		and 
		\begin{equation}\label{I2}
			I_2=I_2(\chi, R, T) := \int_\mathbb R L(1/2+it, \chi) |R(t)|^2 \, \Phi(t/T) \, dt.  
		\end{equation}
		
		We now establish the following estimates for $I_1$ and $I_2$.
		
		\begin{lem}\label{lem:I1 bound}
			For sufficiently large $T$ and \(N=\lfloor T^{1/4}\rfloor\), we have 
			\[
			I_1 \ll T \sum_{n\le N} |r(n)|^2. 
			\]
		\end{lem}
		\begin{proof}
			The expression for $M_1$ following \cite[Eq.~(14)]{aistleitner2017large} gives
			\begin{align*}
				T^{-1/4}\int_\mathbb{R} |R(t)|^2 \Phi(t/T) dt &= T^{3/4} \hat{\Phi}(0) (1+O(1/T)) \sum_{n\le N} |r(n)|^2 
				\\
				&\ll T^{3/4}  \sum_{n\le N} |r(n)|^2. 
			\end{align*}
			Combining this with Lemma \ref{lem:AFE Z chi}, one  sees that it is enough to prove the desired upper bound for 
			\begin{align}
				& \int_\mathbb{R} c_\chi  \sum_{k \le \sqrt{ q t/2\pi  }} \frac{\chi(k)}{\sqrt{k}}  \exp\left( i t\left( \frac{1}{2}\log\left(\frac{qt}{2\pi e}\right)  - \log k \right)\right) |R(t)|^2 \, \Phi(t/T) \,dt \nonumber
				\\
				& \ll \sum_{n,m \le N} |  r(n)r(m) |  \sum_{k \le \sqrt{qT/\pi}} \frac{1}{\sqrt{k}}  \left |   \int_{\max(T, 2 \pi k^2 / q)}^{2T} \exp\left( i t\left( \frac{1}{2}\log\left( \frac{qt}{2\pi e} \right)  \right)\right) \left( \frac{m}{nk}\right)^{it} \Phi(t/T) \,dt              \right |.  
				\label{expr for I1}
			\end{align}
			The same second-derivative estimate below holds uniformly on every subinterval of $[T,2T]$, including
			$[a_k,2T]$, where $a_k= \max(T, 2 \pi k^2 / q)$. 
			Assume that, for $T\leq y\leq 2T$,
			\[
			\mathcal{I}(y)= \int_{T}^{y} \exp\left( i t\left( \frac{1}{2}\log\left( \frac{qt}{2\pi e} \right)  \right)\right) \left( \frac{m}{nk}\right)^{it}  \,dt.    
			\]
			Now, observing that 
			\[
			\frac{d^2}{dt^2} \left( t\left( \frac{1}{2}\log\left( \frac{qt}{2\pi e} \right) 
			+ \log \left ( \frac{m}{nk} \right) \right)\right) \gg \frac{1}{t} \gg \frac{1}{T}, 
			\]
			for $T\le t\le y $. So,   we apply Lemma \ref{lem:sec der test}  to obtain $\mathcal{I}(y) \ll \sqrt{T}$.
			Applying integration by parts gives the following estimate, and replacing the lower endpoint $T$ by $a_k$ gives the same bound uniformly in $k$.
			\begin{align*}
				& \int_{T}^{2T}  \exp\left( i t\left( \frac{1}{2}\log\left( \frac{qt}{2\pi e} \right)  \right)\right) \left( \frac{m}{nk}\right)^{it} \Phi(t/T) \,dt \\
				& =  \left [  \Phi (t/T) \mathcal{I}(t) \right ]_{T}^{2T} - \frac{1}{T}\int_T^{2T} \mathcal{I}(t) \Phi'(t/T) \,dt \\
				& \ll \sqrt{T} +  \sqrt{T} \int_1^2 | \Phi'(x) | \,dx \ll \sqrt{T},
			\end{align*}
			where we have used the fact that  $\Phi$ is a smooth function compactly supported on $[1,2]$,  and hence  $ \int_1^2 | \Phi'(x) |  \,dx <\infty$.
			Therefore, from \eqref{expr for I1}, we have 
			\begin{align*}
				I_1 &  \ll \sqrt{T} \sum_{n,m \le N} |  r(n)r(m) |  \sum_{k \le \sqrt{qT/\pi}} \frac{1}{\sqrt{k}} \\
				& \ll N T^{3/4} \sum_{n \leq N}|r(n)|^2 \ll T \sum_{n \leq N}|r(n)|^2,
			\end{align*}
			by using \(N=\lfloor T^{1/4}\rfloor\). This completes the proof.
		\end{proof}

		\begin{lem}\label{lem:I2 bound}
		%	\normalcolor
		%	\color{blue}
			For sufficiently large $T$ and \(N=\lfloor T^{\alpha}\rfloor\), we have 
			\[
			I_2 \gg T \exp \left( (\sqrt{1-\eta}+o(1)) \sqrt { \frac{\log N  } { \log\log N  } } \right) \sum_{n \le N} |r(n)|^2. 
			\]
		%	\normalcolor
		\end{lem}
		\begin{proof}
			Choose $X= T^{1+\epsilon}$ for some positive $\epsilon$. 
			Using $(14)$ of  \cite{aistleitner2017large}, we have  
			\begin{align}
				I_2 &:= \int_\mathbb R L(1/2+it, \chi) |R(t)|^2 \,  \Phi(t/T) \,  dt 
				\nonumber\\
				&=  
				T \sum_{n,m\le N} \sum_{k\le T^{1+2\epsilon}  }
				\frac{\chi(k) r(n)\overline{r(m) } }{\sqrt{k}} e^{-k/X} \hat{\Phi}\big( T \log(nk/m)\big) + O\left( \left| \int_\mathbb{R} |R(t)|^2 \, \Phi(t/T) ~dt \right| \right).\label{I2=T..}
			\end{align}
			The error term above can be shown to be (see (2) in \cite{soundararajan2007extreme}) 
			\[
			\ll T \sum_{n \le N}|r(n)|^2.
			\]
			For the contribution from off-diagonal terms $nk\ne m$, we use $\left |\log(nk/m)  \right |\gg \frac{1}{N}$ to see that 
			\[
			\hat{\Phi}\left( T \log\left( \frac{ nk}{m}\right)\right) \ll_B \left(\frac{T}{N}\right)^{-B}.
			\]
			Since $N\leq T^{\alpha}$, $(\alpha<1)$, we obtain that the off-diagonal contribution to \eqref{I2=T..} is 
			\[
			\ll N \sum_{n\le N} |r(n)|^2  \sum_{k\le T^{1+2\epsilon}  }
			\frac{|\chi(k)|}{\sqrt{k}} T^{-A}  \ll T^{-A}\sum_{n \le N}|r(n)|^2,
			\]
			for any $A>1$. Hence, 
			\begin{align*}
				I_2 = T\hat{\Phi}(0) \sum_{nk=m \le N} \frac{\chi(k) r(n)       \overline{r(m)}}{\sqrt k} e^{-k/X} + O\bigg( T\sum_{n \le N}|r(n)|^2\bigg). 
			\end{align*}
			As $e^{-k/X} \ge 1/2$ for $k< T^{1-\epsilon}$ and $ r(n)\overline{r(nk)} = |r(n)|^2\overline{ r(k)  }  \,  \mathbbm{1}_{( n,k)=1}      = |\chi(n)|^2 \overline {\chi(k)} f(n)^2f(k)  \,  \mathbbm{1}_{( n,k)=1}  $, 
			applying Lemma \ref{lemma >=exp()}, we have 
			\begin{align*}
				\sum_{nk=m \le N} \frac{\chi(k) r(n)    \overline{r(m)}}{\sqrt k} e^{-k/X}
				\gg
				\exp \left( ((\sqrt{1-\eta}+o(1))+o(1)) \sqrt {  \frac{ \log N  } { \log\log N } } \right) \sum_{n \le N} |r(n)|^2.  
			\end{align*} 
			This completes the proof. 
		\end{proof}

		\begin{proof}[Proof of Theorem~\ref{thm:S1}]
		%	\normalcolor
		%	\color{blue}
			Set $N=\lfloor T^{1/4}\rfloor$ and given $\epsilon>0$, choose $\eta>0$ sufficiently small that $ \sqrt{(1-\eta)}/2> 1/2-\epsilon$. %\normalcolor
			  Then
			$$
			\sqrt{\frac{\log N}{\log\log N}}
			=
			\left(\frac12+o(1)\right)
			\sqrt{\frac{\log T}{\log\log T}}.
			$$
			Using Lemma \ref{lem:I1 bound}, along with the estimate \cite[Eq.~(2)]{soundararajan2007extreme} 
			\begin{equation}
				\int_\mathbb{R} |R(t)|^2 \Phi(t/T) ~dt = T\hat{\Phi}(0) (1+O(1/T)) \sum_{n\le N}|r(n)|^2, 
			\end{equation}
			we have 
			\begin{equation*}
				\begin{aligned}
					\frac{I_1} { A(T) \int_T^{2T}|R(t)|^2 \Phi(t/T) \, dt }  \ll& \frac{T \sum_{n\leq N}|r(n)|^2}{\,T \,\sum_{n \leq N} |r(n)|^2 \left(1 + O\left(1/T\right)\right) \, A(T)}\\
					\ll& \frac{1}{\,A(T) }\quad \to 0  \quad \text{as} \quad T \to \infty,
				\end{aligned}
			\end{equation*}
			where 
			$$ A(T)={\exp\left(\left ( \frac{1}{2}-\epsilon \right)\sqrt{\frac{\log T}{\log \log T}}\right) }. $$
			Hence, 
			\begin{equation}\label{I1=o}
				I_1=o\left( A(T) \, \int_T^{2T}|R(t)|^2 \Phi(t/T) \, dt \right).
			\end{equation}
			Therefore, from Lemma \ref{lem:I2 bound} and equation \eqref{I1=o}, we see that  hypotheses of Proposition \ref{mahatap prop} are satisfied. Hence, the proof is complete.
		\end{proof}

		\section{Proof of Theorem \ref{thm general}}\label{section proof thm 2}
		We have 
		\begin{align*}
			Z_F(t)  :=  F\left(1/2+it\right) \, H^{-1/2}(1/2+it) .\label{Hardy type Z-function}
		\end{align*} 
		Let us consider 
		\[
		I_F=I_F(F,R,T):= \int_{\mathbb{R}}\, F(1/2+it) \, |R(t)|^2 \, \Phi(t/T)\, dt
		\]
		and 
		\be\label{define JF}
		J_F:=\int_{T}^{2T} Z_F(t)\, |R(t)|^2 \, dt.
		\ee

		We now prove the following estimates for $I_F$ and $J_F$.
		\begin{lem} \label{lemma  IF lower for F}
			Let $F (s)\in\mathbb{S}_{\mathrm{poly}}$ be a function of degree $2$
			satisfying \eqref{normality conjectures}. Then for sufficiently large $T$ and \(N=\lfloor T^{\alpha}\rfloor\), we have 
			\[
			I_F \gg T \exp \left((\sqrt{\kappa( 1-\eta)}+o(1)) \sqrt { \frac{ \log N}{\log \log N} } \right) \sum_{n\leq N}|r(n)|^2.
			\]
		\end{lem}
		
		\begin{proof}
			The proof of this lemma  follows the same lines as that of Lemma \ref{lem:I2 bound}, replacing $r(n)=\chi (n)f(n)$ and  $X=T^{1+\epsilon}$ by
			$r(n)=a_F(n)f(n)$ and $X=T^{2+\epsilon}$, respectively. 
		\end{proof}

		\begin{lem}\label{Lemma upper bound of JF}
			Let $F (s)\in\mathbb{S}_{\mathrm{poly}}$ be a function of degree $2$
			satisfying Hypothesis~\ref{hyp} and $ \Im \xi_F =0$. Then for sufficiently large $T$, we have 
			\[
			J_F = o\left( T \, \sum_{n \leq N}|r(n)|^2\right),
			\]
			where  $N=\lfloor T^{\alpha}\rfloor$ with  $0<\alpha<1-\gamma$.
		\end{lem}
		
		\begin{proof}
			Using the definition of $R(t)$ 	in \eqref{resonator}, we write 
			\[
			J_F =  \sum_{m,n \leq N}  r(n)\overline{r(m)}\,K_{m,n} \qquad \text{(say)},
			\]
			where
			\[
			K_{m,n}:= \int_{T}^{2T} Z_F(t)  \left(\frac mn \right)^{it} \, dt .
			\]
			As $Z_F(t) \in \mathbb{R}$, we get $\overline{K}_{m,n} = K_{n,m} .$ 	Therefore, the contribution to $J_F$ from the non-diagonal terms for which $m\neq n$ is
			\begin{align*}
				& \sum_{m>n} r(n)\overline{r(m)}\, K_{m,n} + \sum_{n>m} r(n)\overline{r(m)}\, K_{m,n}\\
				= &  \sum_{m>n} r(n)\overline{r(m)}\, K_{m,n} +  \overline{  \sum_{n>m} r(m) \overline{r(n)} \, K_{n,m} } \\
				= & 2 \Re{ \left (  \sum_{m>n} r(n) \overline{r(m)}\, K_{m,n}   \right)}.
			\end{align*}
			Hence, 
			\[
			J_F =  \sum_{n \leq N} |r(n)|^2 K_{n,n} +   2 \Re{ \left (  \sum_{1\leq n < m \leq N} r(n) \overline{r(m)}\, K_{m,n}   \right)}.
			\]	
			Therefore, 
			\begin{align}
				|J_F|  \, & \leq  \,   2   \sum_{1\leq n \leq m \leq N} \left |  r(n) r(m) \right | | K_{m,n}|
				\nonumber \\   
				\, &  \leq \,  2  \sum_{1\leq n \leq m \leq N}    \frac{ \left |  r(n) r(m) \right |}{\sqrt{{m}/{n}}} \left |   \int_{T}^{2T} F\left( \frac12+it \right ) H^{-\frac12}\left (  \frac12+it  \right)  \left(\frac mn \right)^{\frac12+it}  dt \right|\label{JF}.
			\end{align}	
			We now turn to establishing an upper bound for the integral
			\begin{align*}
				I_{m,n} & =\int_{T}^{2T} F\left( \frac12+it \right ) H^{-\frac12}\left (  \frac12+it  \right)  \left(\frac mn \right)^{\frac12+it}  dt\\
				& = -i \int_{\frac{1}{2}+i T}^{\frac{1}{2}+i2T} F(s)  H^{-\frac12}(s)  \left(\frac mn \right)^{s}  ds.
			\end{align*}
			We apply Cauchy's theorem to estimate the integral 
			\[
			\int_{\mathcal{R}}  F(s) H^{-\frac12}(s) \left( \frac{m}{n} \right)^s  ds,
			\]
			where $\mathcal{R}$ denotes the rectangle with vertices at
			$\frac{1}{2} + iT$, $1+\delta + iT$, $1+\delta + i2T$, and $\frac{1}{2} + i2T$.
			Here $\delta > 0$  is a small constant to be specified later. By Cauchy's theorem, we have 
			\begin{align}\label{I=J1 J2 J3}
				I_{m,n}=-i (J_1+J_2+J_3),
			\end{align}
			where we have dropped the $m,n$ suffix for each $J_i$ for ease of notation, and 
			\begin{align*}
				J_1 &= \int_{1+\delta}^{\frac{1}{2}} F\left( \sigma +i2T \right ) H^{-\frac12}\left (  \sigma+i2 T \right)  \left(\frac mn \right)^{\sigma+i2T}  d\sigma,  \\ 
				J_2  &= i \int_{T}^{2T} F\left(1+\delta +it \right ) H^{-\frac12}\left ( 1+\delta+it \right)  \left(\frac mn \right)^{1+\delta+it}  dt, \\
				J_3 &= \int_{\frac{1}{2}}^{1+\delta} F\left( \sigma +iT \right ) H^{-\frac12}\left (  \sigma+i T \right)  \left(\frac mn \right)^{\sigma+iT}  d\sigma.
			\end{align*}
			We now first estimate $J_3$. By using the estimate given in \eqref{esti H -1/2 and F}, we have 
			\begin{align}\label{J3 error}
				J_3\ll T^{\frac{1}{2}+\delta} \times
				\begin{cases}
					1& \text{if } m = n, \\
					\left( \dfrac{m}{n} \right)^{1+\delta} & \text{if } m > n.
				\end{cases} 
				\qquad  \ll  T^{\frac{1}{2}+\delta} \left( \dfrac{m}{n} \right)^{1+\delta}.
			\end{align}
			The bound for $J_1$ follows similarly to that for $J_3,$ giving the same estimate as in \eqref{J3 error}. We now estimate $J_2$. Because of the asymptotic formula \eqref{Hs -1/2}, $J_2$ is a constant multiple of  
			\begin{align}\label{J2= c1} 
				\left(\frac{m}{n}\right)^{1+\delta}  \int_{T}^{2T} \sum_{k=1}^{\infty} \frac{a_F(k)}{k^{1+\delta+it}} \left( Q_ot \right)^{\frac 12 +\delta+it} e^{-it} \left(\frac{m}{n}\right)^{it}  (1+O(1/t))\, dt.
			\end{align}
			The error term in \eqref{J2= c1} gives 
			
			\begin{equation}\label{J2 first error}
				\left(\frac{m}{n}\right)^{1+\delta}  \int_{T}^{2T} O(t^{\delta-1/2}) \, dt = O\left( \left(\frac{m}{n}\right)^{1+\delta}   T^{\frac12+\delta}\right).
			\end{equation}
			Thus we need to estimate 
			\begin{align}
				\left(\frac{m}{n}\right)^{1+\delta} \sum_{k=1}^\infty \frac{a_F(k)}{k^{1 + \delta}} 
				\int_T^{2T} g(t) e^{i f(t)} dt.\label{main integral of I}
			\end{align}
			where $g(t)=t^{\frac{1}{2}+\delta}$ and $f(t)=t \log \left( \frac{Q_ot  m/n}{ek}\right).$

			\noindent
			We now  partition the sum in \eqref{main integral of I} according to the following intervals:
			\[
			1\leq k\leq Q_oT \frac mn- 1, \quad 
			Q_oT\frac mn+1 \leq k\leq 2Q_o T\frac mn - 1, \quad 
			2Q_o T\frac mn + 1 \leq k,
			\]
			with $S_1$, $S_2$, and $S_3$ denoting the sums over these intervals, respectively. We will use the Lemma \ref{lemma tith 1st der} to estimate $S_1$,$S_3$ and to estimate $S_2$ we will use Lemma \ref{lemma three deriv expon}.
			At most four integers are not covered by the intervals considered above. Any such $k$ lies in the range
			$T\frac{m}{n}\ll k\ll T\frac{m}{n},$ and satisfies
			$a_F(k)\ll k^{\epsilon}.$ Therefore, their 
			contribution to \eqref{main integral of I} is 
			\begin{align}
				O\left(  \left(\frac{m}{n}\right)^{1+\delta}  \frac{|a_F(k)|}{k^{1 + \delta}} T^{\frac{1}{2}+\delta}T\right) 
				= O\left(    \left(\frac{m}{n}\right)^{1+\delta}  \frac{k^{\epsilon}}{k^{1+ \delta}} T^{\frac{1}{2}+\delta}T   \right) 
				= O\left(T^{\frac{1}{2}+\epsilon}  \left(\frac mn\right)^{\epsilon}\right).\label{O term of I2}
			\end{align}
			We next estimate the contribution from $S_1$ using
			Lemma~\ref{lemma tith 1st der}. We note that
			$f'(t)=\log\left(\frac{Q_0tm/n}{k}\right)
			\gg \log\left(\frac{Q_0Tm/n}{k}\right)$ and
			$g(t)\asymp T^{\frac12+\delta}$.
			If necessary, we split $[T,2T]$ into two subintervals so that the required monotonicity condition for $f'/g$ is satisfied. Hence, Lemma~\ref{lemma tith 1st der} gives
			\begin{align}\label{S1}
				S_1   \ll   \left(\frac{m}{n}\right)^{1+\delta}   \sum_{1\leq k \leq TQ_o \frac{m}{n} - 1} \frac{|a_F(k)|}{k^{1+\delta}}  \frac{T^{\frac{1}{2}+\delta}}{\log\left(\frac{TQ_o \frac mn}{ k }\right)}.  
			\end{align}
			We split the above sum once again into two parts:
			\begin{align*}
				1\leq k \leq \frac{Q_oTm/n}{2} \quad \text{and} \quad  \frac{Q_oTm/n}{2}<k \leq Q_oTm/n-1,
			\end{align*}
			and denote them by $S_{11}$ and $S_{12}$, respectively.
			We first consider $S_{11}$. For
			$1\leq k\leq \frac{Q_0Tm/n}{2}$, we have
			$\log\left(\frac{Q_0Tm/n}{k}\right)\geq \log 2$.
			Hence, we obtain
			
			\begin{align}\label{S11}
				S_{11}\ll   \left(\frac{m}{n}\right)^{1+\delta} T^{\frac12+\delta} \sum_{1\leq k \leq \frac{Q_oTm/n}{2}} \frac{|a_F(k)|}{k^{1+\delta}} \ll  \left(\frac{m}{n}\right)^{1+\delta} T^{\frac12+\delta},
			\end{align}
			as $\sum_{k=1}^{\infty} a_F(k)k^{-s}$ converges absolutely for $\Re(s)>1$.

			Now, we estimate $S_{12}$, using the inequality  $\log \frac{Q_oTm/n}{k} \geq \frac{Q_oTm/n-k}{Q_oTm/n}$ for  $ \frac{Q_oTm/n}{2}< k \leq Q_oTm/n-1$ and $a_F(k)=O(k^{\epsilon})$,  we have
			\begin{equation}\label{S12}
				\begin{aligned}
					S_{12} \ll & \left(\frac{m}{n}\right)^{1+\delta} T^{\frac12+\delta} \sum_{\frac{Q_oTm/n}{2}<k \leq Q_oTm/n-1} \frac{|a_F(k)|}{k^{1+\delta}}\frac{Q_oTm/n}{Q_oTm/n-k} \\
					\ll &  \left(\frac{m}{n}\right)^{1+\epsilon}  T^{\frac12+\epsilon}  \sum_{\frac{Q_oTm/n}{2}<k \leq Q_oTm/n-1} \frac{1}{Q_oTm/n-k}.
				\end{aligned}
			\end{equation}
			Estimating the sum in \eqref{S12}, we observe that
			\begin{align*}
				\sum_{\frac{Q_oTm/n}{2}<k \leq Q_oTm/n-1} \frac{1}{Q_oTm/n-k}= \sum_{1\leq l< \frac{Q_oTm/n}{2} } \frac{1}{l+g}  \ll\log (Tm/n),
			\end{align*}
			where $g$ is the fractional part of $Q_oTm/n$. Using this, from \eqref{S12}, we have 
			\begin{align}\label{final esti S12}
				S_{12} \ll  \left(\frac{m}{n}\right)^{1+\epsilon}  T^{\frac12+\epsilon}\log (Tm/n).
			\end{align}
			Therefore,  from  \eqref{S11} and \eqref{final esti S12}, we obtain that 
			\begin{align}\label{final est S1}
				S_1\ll \left(\frac{m}{n}\right)^{1+\delta} T^{\frac12+\delta}+\left(\frac{m}{n}\right)^{1+\epsilon}  T^{\frac12+\epsilon}\log (Tm/n).
			\end{align}
			The term $S_3$ is handled in the same way as $S_1$. In this case, we apply Lemma~\ref{lemma tith 1st der} together with the bound
			$|f'(t)|\gg \log\left(\frac{k}{2Q_0Tm/n}\right)$.
			Hence, by the same argument, we obtain
			\be \label{final est S3}
			S_3\ll \left(\frac{m}{n}\right)^{1+\delta} T^{\frac12+\delta}+\left(\frac{m}{n}\right)^{1+\epsilon}  T^{\frac12+\epsilon}\log (Tm/n).
			\ee
			Finally, we estimate the sum $S_2$ by using Lemma \ref{lemma three deriv expon}. Here note that $c=k/ {( Q_om/n)},\lambda_2=1/T,\lambda_3=1/T^2,G=T^{1/2+\delta}$ and $G_1=T^{-1/2+\delta}$. Hence, we have 
			\beas \label{sum S2}
			S_2 &=& c_1 \sum_{TQ_o\frac mn+1 \leq k\leq 2TQ_o \frac mn - 1} a_F(k) e^{-\frac{ik}{Q_om/n}}+ O \left( \left(\frac{m}{n}\right)^{1+\delta} T^{\frac12+\delta}  \sum_{TQ_o\frac mn+1 \leq k\leq 2TQ_o \frac mn - 1}\frac{|a_F(k)|}{k^{1+\delta}} \right) \\      
			&+& O\left( \left(\frac{m}{n}\right)^{1+\delta} T^{\frac12+\delta}  \sum_{TQ_o\frac mn+1 \leq k\leq 2TQ_o \frac mn - 1}\frac{|a_F(k)|}{k^{1+\delta}}\left( \frac{1}{\log (\frac{k}{Q_oTm/n})} + \frac{1}{\log (\frac{2Q_oTm/n}{k})} \right)\right), 
			\eeas
			where $c_1$ is a constant.
			
			The absolute convergence of
			$\sum_{k=1}^{\infty} a_F(k)k^{-s}$ for $\Re(s)>1$
			shows that the first error term in $S_2$ is
			$O\left(\left(\frac{m}{n}\right)^{1+\delta}T^{1/2+\delta}\right).$
			The second error term can be treated in the same way as $S_{12}$, which gives
			$  
			O\left(
			\left(\frac{m}{n}\right)^{1+\varepsilon}
			T^{1/2+\varepsilon}
			\log(Tm/n)
			\right).$ 
			Thus, it remains only to bound the exponential sum occurring in $S_2$.
			We have
			\begin{equation}
				S_2 = c_1 \sum_{TQ_o\frac mn+1 \leq k\leq 2TQ_o \frac mn - 1} a_F(k) e^{-\frac{ik}{Q_om/n}} +  O \left( \left(\frac{m}{n}\right)^{1+\delta} T^{\frac12+\delta}\right)+  O\left( \left(\frac{m}{n}\right)^{1+\epsilon}  T^{\frac12+\epsilon}\log (Tm/n) \right).
			\end{equation}
			Hence, using the estimate of $S_1,S_2, \text{ and }S_3$, we get
			\begin{equation}
				J_2 \ll \left | \sum_{TQ_o\frac mn+1 \leq k\leq 2TQ_o \frac mn - 1} a_F(k) e^{-\frac{ik}{Q_om/n}} \right |+   \left(\frac{m}{n}\right)^{1+\delta} T^{\frac12+\delta}+  \left(\frac{m}{n}\right)^{1+\epsilon}  T^{\frac12+\epsilon}\log (Tm/n).
			\end{equation}
			Using the estimate of $J_1,J_2 \text{  and } J_3$, from \eqref{I=J1 J2 J3}  we have
			\begin{equation}\label{bound for I}
				I_{m,n} \ll  \left | \sum_{TQ_o\frac mn+1 \leq k\leq 2TQ_o \frac mn - 1} a_F(k) e^{-\frac{ik}{Q_om/n}} \right | + \left(\frac{m}{n}\right)^{1+\delta} T^{\frac12+\delta}+     \left(\frac{m}{n}\right)^{1+\epsilon}  T^{\frac12+\epsilon}\log (Tm/n).
			\end{equation}
			Therefore, from \eqref{JF}, with
			$Q_0=\sqrt{q_F}/(2\pi)$, we obtain
			\begin{equation}\label{Final JF}
				\begin{aligned}
					J_F \ll \sum_{1\leq n \leq m \leq N}
					\frac{|r(m)r(n)|}{\sqrt{\frac{m}{n}}}
					\Bigg(
					&\left|
					\sum_{TQ_o\frac{m}{n}+1 \leq k\leq 2TQ_o\frac{m}{n}-1}
					a_F(k)e^{-\frac{2\pi i k}{\sqrt{q_F}m/n}}
					\right| \\
					&+\left(\frac{m}{n}\right)^{1+\delta}T^{\frac12+\delta}
					+\left(\frac{m}{n}\right)^{1+\epsilon}
					T^{\frac12+\epsilon}\log(Tm/n)
					\Bigg).
				\end{aligned}
			\end{equation}
			Next, we use Hypothesis~\ref{hyp} to estimate the exponential sum in \eqref{Final JF}. We have
			\[
			\sum_{TQ_0\frac{m}{n}+1\leq k\leq 2TQ_0\frac{m}{n}-1}
			a_F(k)e^{-\frac{2\pi i k n}{m\sqrt{q_F}}}
			\ll
			\left(T\frac{m}{n}\right)^{\gamma}
			\left(\log\left(T{m}/{n}\right)\right)^{\gamma_1}.
			\]
			Hence, from \eqref{Final JF}, we obtain 
			\begin{align}
				J_F & \ll \sum_{1\leq n \leq m \leq N}   |r(m)r(n)|\left(  \left(\frac{m}{n}\right)^{\frac12+\epsilon} T^{\frac12+\epsilon}\log (Tm/n)+  \left(\frac{m}{n}\right)^{\frac12+\delta} T^{\frac12+\delta}  + T^{\gamma} \left(\frac{m}{n}\right)^{\gamma-\frac12} \left ( \log (T m/n) \right)^{\gamma_1} \right) \nonumber\\
				& \ll   \sum_{1\leq n \leq m \leq N}   |r(m)r(n)|\left(   T^{\gamma} \left(\frac{m}{n}\right)^{\gamma-\frac12} \left ( \log (T m/n) \right)^{\gamma_1}+  \left(\frac{m}{n}\right)^{\frac12+\delta} T^{\frac12+\delta}  \right),\label{JF<=..} 
			\end{align}
			where we use the fact that, for $0<\epsilon<\delta$, 
			\[
			T^{\frac12+\epsilon} \left(\frac{m}{n}\right)^{\frac12+\epsilon} \log(T m/n) \ll T^{\frac12+\delta} \left(\frac{m}{n}\right)^{\frac12+\delta}.
			\]
			Now, define
			$$
			S_a=\sum_{1\leq n\leq m\leq N}|r(m)r(n)|
			\left(\frac{m}{n}\right)^a.
			$$
			By the Cauchy--Schwarz inequality,
			$$
			S_a
			\leq
			\left(\sum_{m\leq N}m^{2a}\right)^{1/2}
			\left(\sum_{n\leq N}n^{-2a}\right)^{1/2}
			\sum_{n\leq N}|r(n)|^2.
			$$
			Consequently,
			\be\label{cauchy}
			S_a\ll N\sum_{n\leq N}|r(n)|^2
			\quad (\lvert a\rvert<1/2),
			\qquad
			S_{1/2+\delta}\ll
			N^{1+\delta}\sum_{n\leq N}|r(n)|^2
			\quad (\delta>0).
			\ee
			Using \eqref{cauchy} in \eqref{JF<=..}, we obtain
			\begin{align}\label{JF<< final}
				J_F \ll  \left( T^{\gamma} N (\log T)^c + T^{\frac12+\delta} N^{1+\delta} \right)   \sum_{n\leq N}|r(n)|^2, \qquad c=\max ( \gamma_1, 0).
			\end{align}
			Choose $N=\lfloor T^{\alpha}\rfloor$ with any fixed $\alpha<1-\gamma$, and choose
			$$
			0<\delta < \frac{1/2-\alpha}{1+\alpha}.
			$$
			Then $\alpha<1/2$, $\gamma+\alpha<1/2$, and $1/2+\delta+\alpha(1+\delta)<1$. Thus, from \eqref{JF<< final}, we get $J_F=o\left(T\sum_{n\leq N}|r(n)|^2\right)$. This completes the proof.
		\end{proof}

		\begin{proof}[Proof of Theorem~\ref{thm general}]
		%	\normalcolor
		%	\color{blue}
			Set $N=\lfloor T^{\alpha}\rfloor$, with $\alpha<1-\gamma$, and given $\epsilon>0$, choose $\eta>0$ sufficiently small that $ \sqrt{\kappa(1-\eta)}> \sqrt{\kappa}-\epsilon$. %\normalcolor 
			Then
			$$
			\sqrt{\frac{\log N}{\log\log N}}
			=
			\left(\sqrt{\alpha}+o(1)\right)
			\sqrt{\frac{\log T}{\log\log T}}.
			$$
			Thus,  Lemmas~\ref{lemma IF lower for F} and \ref{Lemma upper bound of JF} give
			$$
			I_F\gg A(T)\,T\sum_{n\leq N}|r(n)|^2,
			\qquad
			J_F=o\left(T\sum_{n\leq N}|r(n)|^2\right),
			$$
			where
			$$
			A(T)=
			\exp\left(
			\left(\sqrt{\kappa(1-\gamma)}-\epsilon\right)
			\sqrt{\frac{\log T}{\log\log T}}
			\right).
			$$
			Since $|Z_F(t)|=|F(1/2+it)|$ and $0\leq \Phi \leq 1$, we have
			$$
			\int_T^{2T}|Z_F(t)||R(t)|^2\,dt
			\geq |I_F|
			\gg A(T)\,T\sum_{n\leq N}|r(n)|^2.
			$$
			Also, by \eqref{int Rt},
			$$
			\int_T^{2T}|R(t)|^2\,dt
			\sim T\sum_{n\leq N}|r(n)|^2.
			$$
			Using
			$$
			2\int_T^{2T}Z_F^+(t)|R(t)|^2\,dt
			=
			\int_T^{2T}|Z_F(t)||R(t)|^2\,dt+J_F,
			$$
			we obtain
			$$
			\int_T^{2T}Z_F^+(t)|R(t)|^2\,dt
			\gg
			A(T)\,T\sum_{n\leq N}|r(n)|^2.
			$$
			Dividing by $\int_T^{2T}|R(t)|^2\,dt$, we get
			$$
			\max_{T\leq t\leq 2T}Z_F^+(t)
			\gg
			\exp\left(
			\left(\sqrt{\kappa(1-\gamma)}-\epsilon\right)
			\sqrt{\frac{\log T}{\log\log T}}
			\right).
			$$
			The same argument gives the corresponding result for $Z_F^-(t)$. This completes the proof.
		\end{proof}

	%	\normalcolor
	%	\color{blue}
		\begin{proof}[Proof of Theorem~\ref{thm:S2}]

			The proof of this theorem is similar to that of Theorem~\ref{thm general}. From Lemmas~\ref{lemma exponential sum} and \ref{normality cojecture for both Maass and modular}, we have $\gamma=1/2$ and $\kappa=1$. Hence, the constant $C_F$ is
			$
			{1}/{\sqrt{2}}-\epsilon,
			$
			where $\epsilon>0$ is arbitrarily small.
			
		\end{proof}
%	\normalcolor

		%%%%%%%%%%%%%%%
		\subsection{The unconditional Hecke--Maass case}
		\label{subsec:unconditional-Maass}
		
		We now remove the Ramanujan hypothesis from the Hecke--Maass case of
		Theorem~\ref{thm:S2}. Let 
		$ \theta={7}/{64}.
		$
		By the Kim--Sarnak bound, for every $\epsilon>0$,
		\be\label{ramanujan hyp for maass}
		a_F(n)\ll n^{\theta+\epsilon}.
		\ee
		Moreover, by Rankin--Selberg theory, the Dirichlet series associated
		with $F(s)$ converges absolutely in the half-plane $\Re(s)>1$.

		%\color{blue}
		
		\begin{proof}[Proof of Theorem~\ref{thm:Maass-unconditional}]

		In the proof of Lemma~\ref{lemma  IF lower for F}, the off-diagonal
		contribution is estimated using the rapid decay of the Fourier
		transform $\widehat{\Phi}$. In the present case, we use the same
		argument together with \eqref{ramanujan hyp for maass}. Since
		\[
		\widehat{\Phi}(y)\ll_B |y|^{-B}
		\]
		for any fixed $B>0$, we choose $B$ sufficiently large so that the bound \eqref{ramanujan hyp for maass}
		is sufficient to obtain the required off-diagonal estimate.
		Hence, we have  
		\begin{align*}
			I_F = T\hat{\Phi}(0) \sum_{nk=m \le N} \frac{a_F(k) r(n)       \overline{r(m)}}{\sqrt k} e^{-k/X} + O\bigg( T\sum_{n \le N}|r(n)|^2\bigg). 
		\end{align*}

			As $e^{-k/X} \ge 1/2$ for $k< T^{1-\epsilon}$ and $a_F(k) r(n)\overline{r(nk)} = |a_F(k)|^2 |a_F(n)|^2 f(n)^2 f(k)$, 
		applying Lemmas \ref{lemma 4th moment}, \ref{lemma >=exp()} and  we proof Lemma \ref{lemma  IF lower for F}, we have 
		\[
		I_F \gg T \exp \left((\sqrt{\kappa( 1-\eta)}+o(1)) \sqrt { \frac{ \log N}{\log \log N} } \right) \sum_{n\leq N}|r(n)|^2.
		\]

		Now,  proceeding as in the proof of Lemma~\ref{Lemma upper bound of JF}, and
		using \eqref{ramanujan hyp for maass} in place of the Ramanujan
		hypothesis, we obtain, by choosing $\delta>0$ sufficiently small, for
		\[
		N=\lfloor T^\alpha\rfloor
		\quad \text{with} \quad
		\alpha<\frac{1-2\theta}{2(1+\theta)},
		\]
		that
		\[
		J_F=o\left(
		T\sum_{n\leq N}|r(n)|^2
		\right).
		\]
		Hence, the argument used in the proof of Theorem~\ref{thm general}
		carries over to the present case with the above modifications.
		This completes the proof of
		Theorem~\ref{thm:Maass-unconditional}.
		
		\end{proof}
	%	\normalcolor
		%%%%%%%%%%%%%%%%%

		\section{Proof of Theorem~\ref{thm for dedekind zeta function}}\label{proof for dedekind}
		We recall a few basic facts about the Dedekind zeta function. Let $K$ be a number field. For $\Re(s)>1$, its Dedekind zeta function is defined by
		\begin{equation}\label{full dedekind}
			\zeta_K(s)
			=
			\sum_{\mathfrak{a}\neq 0}\frac{1}{N(\mathfrak{a})^s}
			=
			\sum_{n=1}^{\infty}\frac{a_K(n)}{n^s},
		\end{equation}
		where the first sum runs over the non-zero integral ideals $\mathfrak{a}$ of $K$, $N(\mathfrak{a})$ denotes the absolute norm of $\mathfrak{a}$, and $a_K(n)$ is the number of integral ideals of $K$ having norm $n$.
		
		Now, let $K=\mathbb{Q}(\sqrt{D_K})$ be a quadratic field, where $D_K$ is the fundamental discriminant of $K$. In this case, the Dedekind zeta function factors as
		$$
		\zeta_K(s)=\zeta(s)L(s,\chi_{D_K}),
		$$
		where $\chi_{D_K}(n)=\left(\frac{D_K}{n}\right)$ is the primitive quadratic Dirichlet character associated with $K$. Comparing the Dirichlet series coefficients gives
		\be\label{aK=chiD}
		a_K(n)=\sum_{d\mid n}\chi_{D_K}(d).
		\ee
		We also consider the product $F(s)=L(s,\chi_1)L(s,\chi_2)$, where $\chi_1$ and $\chi_2$ are primitive Dirichlet characters modulo $q_1$ and $q_2$, respectively, and $q_1q_2$ is not a perfect square. Writing $F(s)=\sum_{n=1}^{\infty}a_F(n)n^{-s}$, we have $$a_F(n)=\sum_{uv=n}\chi_1(u)\chi_2(v).$$
		
		To prove Theorem~\ref{thm for dedekind zeta function}, we need a uniform upper bound for the exponential sum appearing in \eqref{Final JF}. However, for the Dedekind zeta function $\zeta_K(s)$ over a quadratic field $K$, as well as for the product $L(s,\chi_1)L(s,\chi_2)$, a bound of the required strength is not currently available. For $\zeta_K(s)$, the known estimates are of the weaker form $o(T)$ for irrational $x$, as shown by Chandrasekharan and Narasimhan~\cite{chandrasekharan1968zeta}. This is not sufficient for our argument, since we require an upper bound involving a power of $T$ strictly smaller than $1$. To obtain such an estimate, we use the following lemma.
		\begin{lem}\cite[Lemma~2.2]{vaughan2003hardy}.\label{Vaughan lemma}
			Let $Y$, $Z$, and $\alpha_o$ be real numbers such that $Y,Z\geq 1$, and suppose that
			$$
			\left|\alpha_o-\frac{a}{q}\right|\leq \frac{1}{q^{2}},
			$$
			where $(a,q)=1$. Then
			$$
			\sum_{x\leq Y}
			\min\left(\frac{YZ}{x},\|\alpha_o x\|^{-1}\right)
			\ll
			\left(
			\frac{YZ}{q}
			+ Y
			+ q
			\right)\log(2Yq),
			$$
			where $\|y\|$ denotes distance to the nearest integer.
		\end{lem}

		%\begin{lem}\label{lemma expon dede}
		%	Let $K=\mathbb{Q}(\sqrt{D_K})$ be a  quadratic field with $K\neq\mathbb{Q}(i)$ and $a_K(n)$ be as defined in \eqref{aK=chiD}. Then 
		%	\begin{align}
			%		\sum_{k\le X}a_K(k)e\!\left(-\frac{kn}{m\sqrt{|D_K|}}\right)
			%		\ll_K \sqrt{mnX}\,\log(2mnX),\label{Sum a_K (k) e(-.)}
			%	\end{align}
		%uniformly for positive integers \(m,n\) and \(X\ge2\), and $e(x):=e^{2\pi i x}$. Here, the dependence of the implied constant is only on $K$.		
		%\end{lem}

		We next establish an estimate for the exponential sums corresponding to the Dedekind zeta function $\zeta_K(s)$ over a quadratic field $K$ and to the product $L(s,\chi_1)L(s,\chi_2)$.

		\begin{lem}\label{lemma expon dede}
			Let $F$ be one of the following:
			\begin{itemize}
				\item[(i)] $F(s)=\zeta_K(s)$, where $K=\mathbb{Q}(\sqrt{D_K})$ is any quadratic field with $K\neq\mathbb{Q}(i)$, in which case $q_F=|D_K|$ and $a_F(n)=a_K(n)$ is given by \eqref{aK=chiD};
				
				\item[(ii)] $F(s)=L(s,\chi_1)L(s,\chi_2)$, where $\chi_1$ and $\chi_2$ are primitive Dirichlet characters modulo $q_1$ and $q_2$, respectively, with $q_1q_2$ not a perfect square, in which case $q_F=q_1q_2$ and
				$
				a_F(k)=\sum_{uv=k}\chi_1(u)\chi_2(v).
				$
			\end{itemize}
			Then
			\begin{align}
				\sum_{k\leq X}a_F(k)e\!\left(-\frac{kn}{m\sqrt{q_F}}\right)
				\ll \sqrt{mnX}\,\log(2mnX),
				\label{Sum a_F (k) e(-.)}
			\end{align}
			uniformly for positive integers $m,n$ and $X\geq 2$, where $e(x):=e^{2\pi i x}$. Here, the implied constant depends only on $F$.
		\end{lem}

		\begin{proof}
			Set
			$$
			\beta=\frac{1}{\sqrt{q_F}},\qquad
			\theta=\frac{n}{m}\beta,
			$$
			and 
			$$
			S_{m,n}(X)
			:=\sum_{k\le X}a_K(k)e(-\theta k).
			$$
			
			\medskip

			We first consider case $(i)$:  $F(s)=\zeta_K(s)$. In this case, $q_F=|D_K|=:d$. Since $K\neq\mathbb{Q}(i)$, the integer $d$ is not a perfect square. Hence,
			$\beta$	is a quadratic irrational. Therefore, there exists a constant \(c_K>0\) such that
			$$
			\left|\beta-\frac uv\right|\ge \frac{c_K}{v^2}
			\qquad(u\in\mathbb Z,\ v\ge1)
			$$
			Applying this with the rational number $ma/(nq)$, we obtain
			\begin{align}
				\left|\theta-\frac aq\right|
				&=\frac nm
				\left|\beta-\frac{ma}{nq}\right|\nonumber\\
				&\ge \frac nm\frac{c_K}{n^2q^2}
				=\frac{c_K}{mnq^2}.\label{theta-a/q}
			\end{align}
			The fraction $ma/(nq)$ need not be reduced; reducing it only strengthens the lower bound.
			We also need the same estimate for
			$$
			d\theta=\frac nm\sqrt d.
			$$
			Since \(\sqrt d\) is another fixed quadratic irrational, after decreasing \(c_K\), we have
			\begin{align}
				\left|\vartheta-\frac aq\right|
				\ge \frac{c_K}{mnq^2},
				\qquad
				\vartheta\in\{\theta,d\theta\}.\label{theta or d theta approx}
			\end{align}
			Now we choose
			$$
			\mathcal Q=\left\lceil\sqrt{mnX}\right\rceil.
			$$
			For each of \(\vartheta=\theta,d\theta\), Dirichlet’s approximation theorem gives its own coprime integers \(a,q\) satisfying
			\begin{equation}\label{diri appro}
				1\le q\le\mathcal Q,
				\qquad
				\left|\vartheta-\frac aq\right|
				\le \frac1{q\mathcal Q}.
			\end{equation}
			Combining this with \eqref{theta or d theta approx},
			$$
			\frac{c_K}{mnq^2}\le\frac1{q\mathcal Q},
			$$
			and therefore
			\begin{align}
				\frac{c_K\mathcal Q}{mn}\le q\le\mathcal Q.\label{gene rel of q <Q}
			\end{align}
			In particular, we have 
			\begin{align}
				\frac Xq+q\ll_K\sqrt{mnX}.\label{X/q+q<<}
			\end{align}
			Indeed,
			$$
			\frac Xq\le\frac{mnX}{c_K\mathcal Q}
			\ll_K\sqrt{mnX},
			\qquad
			q\le\mathcal Q\ll\sqrt{mnX}.
			$$
			Using \eqref{aK=chiD}, we have 
			$$
			S_{m,n}(X)
			:=\sum_{k\le X}a_K(k)e(-\theta k)
			=\sum_{uv\le X}\chi_{ D_K}(u)e(-\theta uv).
			$$
			Let \(H=\lfloor\sqrt X\rfloor\). Every pair \(uv\le X\) has either \(u\le H\), or \(u>H\) and \(v\le H\). Thus, we have 
			\begin{align}
				S_{m,n}(X)
				={}&
				\sum_{u\le H}\chi_{ D_K}(u)
				\sum_{v\le X/u}e(-\theta uv) \nonumber\\
				&+
				\sum_{v\le H}
				\sum_{H<u\le X/v}\chi_{ D_K}(u)e(-\theta uv).\label{exact decom}
			\end{align}
			For $\beta_o \notin \mathbb{Z}$ and an interval of $L$ consecutive integers, the standard geometric-sum estimate gives
			\begin{equation}\label{geometric sum}
				\left|\sum_{\ell=A+1}^{A+L} e(\beta_o\ell)\right|
				\ll \min\left\{L,\|\beta_o\|^{-1}\right\},
			\end{equation}
			where $\|\beta_o\|$ denotes  distance to the nearest integer (see \cite[Eq.~(8.6)]{iwaniec2004analytic}).
			
			Since $|\chi_D(u)|\leq 1$, using \eqref{geometric sum}, the first term in \eqref{exact decom} is bounded by
			\begin{align}
				\sum_{u\le H}
				\min\left(\frac Xu,\|u\theta\|^{-1}\right).\label{min X/u}
			\end{align}
			For the second term, fix \(v\) and split \(u\) into residue classes modulo \(d\), we have 
			\begin{align}
				\sum_{H<u\le X/v}\chi_D(u)e(-\theta uv)
				={}
				\sum_{b\bmod d}\chi_D(b)e(-\theta vb)
				\sum_{\substack{\ell\in\mathbb Z\\H<b+d\ell\le X/v}}
				e(-d\theta v\ell).\label{2nd term of exac deco}
			\end{align}
			The inner sum is geometric. Its number of terms is at most \(X/(dv)+1\), so
			\begin{align}
				\left|
				\sum_{H<u\le X/v}\chi_D(u)e(-\theta uv)
				\right|
				\ll_K
				\min\left(\frac Xv,\|dv\theta\|^{-1}\right).\label{min X/v}
			\end{align}
			Here the factor arising from the \(d\) residue classes is absorbed into the constant depending on \(K\). Define
			$$
			V(X,w)
			=\sum_{h\le H}
			\min\left(\frac Xh,\|h w\|^{-1}\right).
			$$
			Then equations \eqref{exact decom}–\eqref{min X/v} give
			\begin{align}
				S_{m,n}(X)
				\ll_K V(X,\theta)+V(X,d\theta).\label{|Smn|}
			\end{align}
			The approximations constructed in \eqref{diri appro} satisfy the hypothesis of Lemma \ref{Vaughan lemma} because
			$$
			q\le\mathcal Q
			\quad\Longrightarrow\quad
			\frac1{q\mathcal Q}\le\frac1{q^2}.
			$$
			Therefore, using Lemma~ \ref{Vaughan lemma}, we have 
			\begin{align}
				\sum_{h\le H}
				\min\left(\frac Xh,\|h\alpha_o\|^{-1}\right)
				\ll
				\left(\frac Xq+H+q\right)\log(2Xq).\label{min X/h}
			\end{align}
			By \eqref{X/q+q<<} and $H \leq \sqrt{X}$, 
			$$
			\frac Xq+H+q\ll_K\sqrt{mnX}.
			$$
			Also \(q\ll\sqrt{mnX}\), so
			$$
			\log(2Xq)\ll\log(2mnX).
			$$
			Apply \eqref{min X/h} separately to \(\theta\) and \(d\theta\). We obtain
			$$
			V(X,\theta),\ V(X,d\theta)
			\ll_K
			\sqrt{mnX}\log(2mnX).
			$$
			Substituting this into \eqref{|Smn|} proves the lemma in case $(i)$.

			\medskip
			
			We now consider case $(ii)$: $F(s)=L(s,\chi_1)L(s,\chi_2)$. In this case, $q_F=q_1q_2$. Since $q_1q_2$ is not a perfect square, $\beta$ is a quadratic irrational. Since
			$
			a_F(k)=\sum_{uv=k}\chi_1(u)\chi_2(v),
			$
			as in \eqref{exact decom}, we have
			\begin{align}
				S_{m,n}(X)
				={}&
				\sum_{u\leq H}\chi_1(u)
				\sum_{v\leq X/u}\chi_2(v)e(-\theta uv) \nonumber\\
				&+
				\sum_{v\leq H}\chi_2(v)
				\sum_{H<u\leq X/v}\chi_1(u)e(-\theta uv).
				\label{exact decom for dirich}
			\end{align}
			Using the periodicity of $\chi_1$ and $\chi_2$, we split the corresponding inner sums into residue classes modulo $q_1$ and $q_2$, respectively. Each resulting sum is a geometric exponential sum of the same type as in case $(i)$. Since $q_1$ and $q_2$ are fixed, the same argument applied to \eqref{exact decom for dirich} gives
			$
			S_{m,n}(X)\ll_{\chi_1,\chi_2}\sqrt{mnX}\log(2mnX).
			$
			This completes the proof of the lemma in case $(ii)$.
		\end{proof}

		\begin{lem}\label{lemma expon dede for K=Q(i)}
			Let $K=\mathbb{Q}(i)$ be a  quadratic field and $a_K(n)$ be as defined in \eqref{aK=chiD}. Then 
			\begin{align}
				\sum_{k\le X}a_K(k)e\!\left(-\frac{kn}{m\sqrt{|D_K|}}\right)
				\ll
				\frac{m}{(m,n)}
				\sqrt{X}
				+
				\frac{m^2}{(m,n)^2},\label{Sum a_K (k) e(-.) for K=Q(i)}
			\end{align}
			uniformly for positive integers \(m,n\) such that $ r(m)r(n) \neq 0$, for sufficiently large $N$, where $r(n)$ is as defined in \eqref{resonator}, and \(X\ge2\). Here, the dependence of the implied constant is only on $K$.		
		\end{lem}

		\begin{proof}
			For  $K=\mathbb{Q}(i)$, we have $D_K=-4$. Assume that 
			\begin{equation}\label{Smn for K=Q(i)}
				{S}_{m,n}(X)
				=
				\sum_{ k\le X }
				a_K(k)
				e\left(-\frac{kn}{2m}\right).
			\end{equation}
			Here, we only need to consider the pairs \(m,n\) for which $
			r(m)r(n)\neq 0.$
			Recall that $ r(n)=a_K(n)f(n),$  where \(f\) is multiplicative defined in \eqref{f(p)}. Since
			$$
			L^2=\frac{1}{2}\log N\log\log N\longrightarrow\infty,
			$$
			we have \(L^2>2\) for sufficiently large \(N\). Hence
			$f(2)=0.$ Thus, if \(2\mid n\), then by the multiplicativity of \(f\), $
			f(n)=0,
			$ 
			and consequently
			$r(n)=0. $
			
			Therefore, for every non-zero contribution to the resonator sum, both \(m\) and \(n\) are odd.
			Let
			$$
			g=(m,n),
			\qquad
			m=gm_1,
			\qquad
			n=gn_1,
			$$
			so that $(m_1,n_1)=1.$ Hence
			$$
			\frac{n}{2m}
			=
			\frac{n_1}{2m_1}
			$$
			is in lowest terms, where $m_1$ and $n_1$ are odd. Put
			$$
			a=n_1,
			\qquad
			q=2m_1=\frac{2m}{(m,n)}.
			$$
			Then $(a,q)=1$ and, since \(m_1\) is odd, $ q\equiv 2\pmod 4.$ For  $K=\mathbb{Q}(i)$, we have
			$$
			a_{K}(k)=\frac{1}{4}r_2(k), 
			$$
			where
			$$
			r_2(k)
			=
			\#\left\{(u,v)\in\mathbb Z^2:u^2+v^2=k\right\}.
			$$
			Therefore, from \eqref{Smn for K=Q(i)},  
			\be\label{4Smn for Q(i)}
			4S_{m,n}(X)
			=
			\sum_{k\le X}
			r_2(k)e\left(-\frac{ak}{q}\right)
			=
			\sum_{\substack{u,v\in\mathbb Z\\
					0<u^2+v^2\le X}}
			e\left(-\frac{a(u^2+v^2)}{q}\right).
			\ee
			We now split \(u\) and \(v\) into residue classes modulo \(q\). For \(r,s\bmod q\), define
			\begin{equation*}
				N_{r,s}(X)
				=
				\#\left\{
				(u,v)\in\mathbb Z^2:
				u^2+v^2\le X,\,
				u\equiv r\pmod q,\,
				v\equiv s\pmod q
				\right\}. 
			\end{equation*}
			Hence,  \eqref{4Smn for Q(i)} gives 
			\begin{equation}\label{4Smn mod$}
				4S_{m,n}(X)
				=
				\sum_{r,s\bmod q}
				e\left(-\frac{a(r^2+s^2)}{q}\right)
				N_{r,s}(X)
				+O(1).
			\end{equation}
			We next estimate \(N_{r,s}(X)\). Write $ u=r+q\ell,
			v=s+qj,\text{ where } \ell,j\in\mathbb Z.$  Then $  u^2+v^2\le X \quad$  is equivalent to
			\begin{equation*}
				\left(\ell+\frac{r}{q}\right)^2
				+
				\left(j+\frac{s}{q}\right)^2
				\le
				\frac{X}{q^2}.
			\end{equation*}
			Thus, \(N_{r,s}(X)\) is the number of integral lattice points in a disk 
			of radius $R={\sqrt X}/{q}$  and center $
			\left(- r/q,- s/q\right).$ The above region is a bounded region, whose boundary is a circle, hence a rectifiable curve.

			Now we use the planar lattice-point estimate of Davenport \cite[p.~179]{daven1951}.   If \(\mathcal R\) is a bounded
			region in \(\mathbb R^2\) whose boundary is a rectifiable curve of
			length \(L(\mathcal R)\), then
			$$
			\left|
			N(\mathcal R)-V(\mathcal R)
			\right|
			<
			4\bigl(L(\mathcal R)+1\bigr),
			$$
			where \(N(\mathcal R)\) denotes the number of integral lattice points
			in \(\mathcal R\) and \(V(\mathcal R)\) denotes its area.
			
			Here, we have
			$$
			V(\mathcal R)=\pi R^2=\frac{\pi X}{q^2} \quad \text{ and } \quad L(\mathcal R)=2\pi R=\frac{2\pi\sqrt X}{q}.
			$$
			Hence, uniformly in \(r,s\),
			$$
			N_{r,s}(X) = \frac{\pi X}{q^2}
			+ O\left(\frac{\sqrt X}{q}+1\right).
			$$
			Consequently, from \eqref{4Smn mod$}, we have 
			\begin{equation}\label{final Smn}
				4S_{m,n}(X)
				=
				\frac{\pi X}{q^2}
				\sum_{r,s\bmod q}
				e\left(-\frac{a(r^2+s^2)}{q}\right)
				+
				O\left(q\sqrt X+q^2\right), 
			\end{equation}
			as there are \(q^2\) pairs \((r,s)\pmod q\).
			Define
			$$
			G(a,q)
			:=
			\sum_{r\bmod q}e\left(\frac{ar^2}{q}\right).
			$$
			Since \(q\equiv 2\pmod 4\), write $q=2h,$ where \(h\) is odd. Moreover, \(a\) is odd. Pairing the terms corresponding to \(r\) and \(r+h\), we have  $(r+h)^2=r^2+2rh+h^2. $ Thus, the terms cancel pairwise, and hence $G(a,q)=0,$  so $G(-a,q)=0,$ as well. Therefore,  the main term  of \eqref{final Smn} vanishes, giving
			$$
			S_{m,n}(X)
			\ll
			q\sqrt X+q^2,
			\qquad
			q=\frac{2m}{(m,n)}.
			$$
			Hence, the proof is complete. 
		\end{proof}

		\begin{lem}\label{JF upper bound for dede}
			Let $F$ be one of the following functions:
			\begin{itemize}
				\item[(i)] $F(s)=\zeta_K(s)$, where $K$ is any quadratic field;
				
				\item[(ii)] $F(s)=L(s,\chi_1)L(s,\chi_2)$, where $\chi_1$ and $\chi_2$ are primitive Dirichlet characters modulo $q_1$ and $q_2$, respectively, with $q_1q_2$ not a perfect square.
			\end{itemize}
			Then, for sufficiently large $T$, we have
			$$
			J_F \ll T \log T	 \sum_{n\leq N}|r(n)|^2,
			$$
			where $J_F$ is defined in \eqref{define JF} and \(N=\lfloor T^{1/4}\rfloor\).
		\end{lem}

		\begin{proof}
			From \eqref{Final JF}, we have 
			\begin{align}\label{JF for dedekind}
				J_F \ll  \sum_{1\leq n \leq m \leq N}  \frac{|r(m)r(n)|}{\sqrt{\frac{m}{n}}} \left( \left |  {S}_{m,n} \right | + \left(\frac{m}{n}\right)^{1+\delta} T^{\frac12+\delta}+  \left(\frac{m}{n}\right)^{1+\epsilon}  T^{\frac12+\epsilon}\log (Tm/n) \right),
			\end{align}
			where 
			\[
			{S}_{m,n}= \sum_{TQ_o\frac mn+1 \leq k\leq 2TQ_o \frac mn - 1} a_K(k)e\!\left(-\frac{kn}{m\sqrt {  |D_K| } }\right).
			\]
			We now divide the proof into two cases.

			\medskip

			Case $1$. Suppose that either $F(s)=\zeta_K(s)$ with $K\neq\mathbb{Q}(i)$, or $F(s)=L(s,\chi_1)L(s,\chi_2)$, where $\chi_1$ and $\chi_2$ are primitive Dirichlet characters modulo $q_1$ and $q_2$, respectively, with $q_1q_2$ not a perfect square.

			In both cases, using Lemma~\ref{lemma expon dede}, we get
			$$
			S_{m,n}
			\ll_F
			\sqrt{mnU}\,\log(2mnU), \qquad \text{where} \qquad U=Q_oT\frac{m}{n}.
			$$
			Thus, uniformly for \(m,n\le N\), we have 
			\begin{align*}
				S_{m,n}
				\ll
				m\sqrt T\,\log(2TN).
			\end{align*}	
			Hence, using \eqref{cauchy}, \eqref{JF for dedekind} gives 
			\begin{align}
				J_F & \ll   \sum_{1\leq n \leq m \leq N}  |r(m)r(n)| \left( T^{\frac12} (mn)^{\frac12} \log (2TN) + \left(\frac{m}{n}\right)^{\frac12+\delta} T^{\frac12+\delta}+  \left(\frac{m}{n}\right)^{\frac12+\epsilon}  T^{\frac12+\epsilon}\log (T N) \right)\nonumber\\
				& \ll  \left (  T^{\frac12} N^2 \log(2TN)+ T^{\frac12+\delta}  N^{1+\delta}  \right) \sum_{n \leq N} |r(n)|^2, 
			\end{align}
			For \(N=\lfloor T^{1/4}\rfloor\), we obtain
			$$
			J_F
			\ll
			\left(
			T\log T
			+
			T^{\frac34+\frac 54\delta}
			\right)\sum_{n \leq N} |r(n)|^2.
			$$

			\medskip

			{Case $2$:} Let $F(s)=\zeta_K(s)$ with $K=\mathbb{Q}(i)$.
			In this case, 
			using Lemma~\ref{lemma expon dede for K=Q(i)}, we get
			$$
			{S}_{m,n}
			\ll
			\frac{m}{(m,n)}
			\sqrt{T\frac{m}{n}}
			+
			\frac{m^2}{(m,n)^2}.
			$$
			Proceeding as in Case~$1$, we obtain
			$$
			J_F
			\ll
			\left(
			T
			+
			T^{\frac34}
			+
			T^{\frac34+\frac 54\delta}
			\right)
			\sum_{n\leq N}|r(n)|^2.
			$$
			In both cases, choosing any fixed $0<\delta<1/5$, we have
			$$
			J_F \ll T \log T	 \sum_{n\leq N}|r(n)|^2.
			$$
			This completes the proof.
		\end{proof}

		\begin{proof}[Proof of Theorem~\ref{thm for dedekind zeta function}]
			
			Set $N=\lfloor T^{1/4}\rfloor$. Let
			$$
			A(T)=\exp\left(C_F\sqrt{\frac{\log T}{\log\log T}}\right).
			$$
			Using Lemmas~\ref{normality conjecture dedekind} and \ref{lemma normality conjecture for prod}, we have $\kappa=2$ in both cases. Since $N=\lfloor T^{1/4}\rfloor$, we get
			$
			C_F={1}/{\sqrt{2}}-\epsilon,
			$
			where $\epsilon>0$ is arbitrarily small.
			
			By Lemma~\ref{JF upper bound for dede} and \eqref{int Rt}, we have
			\begin{align*}
				\frac{J_F}{A(T) \int_T^{2T}|R(t)|^2 \, dt}  \ll \frac{\log T}{A(T)} \qquad \text{as} \qquad T \to \infty,
			\end{align*}
			it follows that
			\be\label{JF=0 dede}
			J_F=o\left(A(T)\int_T^{2T}|R(t)|^2\,dt\right).
			\ee
			Therefore, by Lemma~\ref{lemma IF lower for F} and \eqref{JF=0 dede}, the hypotheses of Proposition~\ref{New prop R(t) phi} are satisfied. This completes the proof.
			
		\end{proof}

		%%%%%%%%%%%%%%%%%%%

		%	\normalcolor
	%	\color{blue}

		\section{Beyond Selberg class}\label{section beyond selberg}

		In this section, we study large positive and negative values of the Hardy functions associated with certain $L$-functions that need not belong to the Selberg class.

		\begin{thm}\label{thm:linear combination selberg}
			Let $F_1,\ldots,F_h$ be  distinct elements of $\mathbb{S}_{\text{poly}}$.  Write
			$$
			F_j(s)=\sum_{n=1}^{\infty}\frac{a_j(n)}{n^s},
			\qquad 1\le j\le h.
			$$
			Assume strong version of Selberg's orthonormality condition that  
			\be\label{strong version of normality}
			\sum_{p\le x}a_j(p)\overline{a_\ell(p)}
			=
			\begin{cases}
				(\kappa+o(1))\dfrac{x}{\log x},
				&j=\ell,\\[6pt]
				o\left(\dfrac{x}{\log x}\right),
				&j\ne\ell,
			\end{cases}
			\ee
			where $\kappa>0$ is fixed. Let
			$\beta_1,\ldots,\beta_h$ be fixed complex numbers,
			not all zero, and put
			$$
			F(s)=\sum_{j=1}^{h}\beta_jF_j(s).
			$$
			Then, for sufficiently large $T$,
			$$
			\max_{T\le t\le2T}
			\left|F\left(\frac12+it\right)\right|
			\gg
			\exp\left(
			(\sqrt{\kappa}+o(1))
			\sqrt{\frac{\log T}{\log\log T}}
			\right).
			$$
		\end{thm}
		
		\begin{proof}
			We follow the  method of Aistleitner
			and Pańkowski \cite[Section~2]{aistleitner2017large},
			and explain the modifications required for the
			linear combination.

			Since the coefficients $\beta_1,\ldots,\beta_h$ are not
			all zero, we may assume, without loss of generality,
			that $\beta_1\ne0$. We choose $r(n)=a_1(n)f(n)$,
			where $f$ is defined in \eqref{f(p)}.
			Since $F(s)=\sum_{j=1}^{h}\beta_jF_j(s)$,
			we have
			$$
			I_F=
			\sum_{j=1}^{h}\beta_j
			\int_{\mathbb R}
			F_j\left(\frac12+it\right)
			|R(t)|^2\Phi(t/T)\,dt.
			$$
			Proceeding as in
			\cite[Section~2]{aistleitner2017large},
			we obtain
			$$
			I_F=
			T\widehat{\Phi}(0)
			\sum_{j=1}^{h}\beta_j\mathcal{M}_j
			+O\left(T\sum_{n\le N}|r(n)|^2\right),
			$$
			where 
			$$
			\mathcal{M}_j=
			\sum_{nk\le N}
			\frac{a_j(k)r(n)\overline{r(nk)}}{\sqrt{k}}
			e^{-k/X},
			$$
			with $X=T^{d+\epsilon}$, where
			$d=\max_{1\le j\le h}d_{F_j}$ and $\epsilon>0$
			is small.
			
			Since $f$ is supported on squarefree integers,
			the nonzero summands satisfy $(n,k)=1$, and
			hence
			$$
			a_j(k)r(n)\overline{r(nk)}
			=
			a_j(k)\overline{a_1(k)}
			|a_1(n)|^2f(n)^2f(k).
			$$
			Thus, unlike the single $L$-function case,
			we need to estimate both the selected
			contribution $\mathcal{M}_1$ and the cross
			contributions $\mathcal{M}_j$, for $j\ne1$.
			
			Put
			$$
			D_\infty=\prod_p(1+u_p),
			\qquad
			P_j=\prod_p(1+u_p+v_{j,p}),
			$$
			where
			$$
			u_p=|a_1(p)|^2f(p)^2,
			\qquad
			v_{j,p}=
			\frac{a_j(p)\overline{a_1(p)}f(p)}{\sqrt p}.
			$$
			By using \eqref{strong version of normality}
			and partial summation,
			$$
			\sum_{L^2\le p\le M}
			\frac{a_j(p)\overline{a_1(p)}}{p\log p}
			=
			\begin{cases}
				\dfrac{\kappa+o(1)}{2\log L},&j=1,\\[6pt]
				o\left(\dfrac1{\log L}\right),&j\ne1.
			\end{cases}
			$$
			Consequently, the Euler-product calculation gives
			$$
			\frac{P_1}{D_\infty}
			=
			\exp\left(
			(\sqrt{\kappa(1-\eta)}+o(1))
			\sqrt{\frac{\log N}{\log\log N}}
			\right),
			$$
			whereas
			$$
			\frac{|P_j|}{D_\infty}
			=
			\exp\left(
			o\left(\sqrt{\frac{\log N}{\log\log N}}\right)
			\right),
			\qquad 2\le j\le h.
			$$
			Applying the same Rankin truncation argument as \cite{aistleitner2017large},
			with absolute values for the cross terms, yields
			$\mathcal{M}_j=P_j+o(D_\infty)$ and
			$\sum_{n\leq N} |r(n)|^2=(1+o(1))D_\infty$.
			Thus $\mathcal{M}_1$ is positive and dominates
			all the cross contributions. Since $\beta_1\ne0$,
			we obtain
			$$
			\begin{aligned}
				|I_F|
				&\ge T\widehat{\Phi}(0)
				\left(
				|\beta_1|\mathcal{M}_1
				-\sum_{j=2}^{h}|\beta_j||\mathcal{M}_j|
				\right)
				-O\left(T\sum_{n\le N}|r(n)|^2\right)\\
				&\gg T
				\exp\left(
				(\sqrt{\kappa(1-\eta)}+o(1))
				\sqrt{\frac{\log N}{\log\log N}}
				\right) \sum_{n\le N}|r(n)|^2.
			\end{aligned}
			$$
			Hence, the inequality
			$$
			\max_{T\le t\le2T}
			\left|F\left(\frac12+it\right)\right|
			\ge
			\frac{|I_F|}
			{\int_{\mathbb R}|R(t)|^2\Phi(t/T)\,dt}
			$$
			give
			$$
			\max_{T\le t\le2T}
			\left|F\left(\frac12+it\right)\right|
			\gg
			\exp\left(
			(\sqrt{\kappa\alpha(1-\eta)}+o(1))
			\sqrt{\frac{\log T}{\log\log T}}
			\right).
			$$
			Choosing $\alpha<1$ sufficiently close to $1$
			and $\eta>0$ sufficiently small completes
			the proof.
		\end{proof}

\begin{thm}\label{thm:general linear combination hardy}
	Let $F(s)$ be as in
	Theorem~\ref{thm:linear combination selberg}.
	Suppose further that all the constituent
	$L$-functions $F_j$ have the same degree
	$d\in\{1,2\}$ and satisfy the functional equations
	$$
	F_j(s)=\omega_jH_0(s)
	\overline{F_j}(1-\overline{s}),
	\qquad 1\le j\le h,
	$$
	where $H_0(s)$ is common to all the
	$L$-functions $F_j$ and $|\omega_j|=1$.
	Assume that
	\be\label{condition on functional equation}
	\beta_j\omega_j=\omega\overline{\beta_j},
	\qquad 1\le j\le h,
	\ee
	for some $\omega\in\mathbb C$ with $|\omega|=1$.
	If $d=2$, suppose additionally that $ \Im \xi_{ F_j} =0$
	and that each
	$F_j$ satisfies Hypothesis~\ref{hyp}, for
	$1\le j\le h$.
	
	Let $Z_F(t)$ denote the corresponding Hardy
	function. Then, for every $\epsilon>0$
	and all sufficiently large $T$,
	$$
	\max_{T\le t\le2T}Z_F^\pm(t)
	\gg
	\exp\left(
	(C_F-\epsilon)
	\sqrt{\frac{\log T}{\log\log T}}
	\right),
	$$
	where 
	$$
	C_F=
	\begin{cases}
		\dfrac12, & d=1,\\[6pt]
		\sqrt{\kappa(1-\gamma)}, & d=2,
	\end{cases}
	$$
	with $\kappa>0$ as in
	Theorem~\ref{thm:linear combination selberg}
	and $\gamma$ as in Hypothesis~\ref{hyp}.
\end{thm}

	\begin{proof}
		By the functional equations and the
		condition  \eqref{condition on functional equation}, we obtain
		$$
		F(s)=\omega H_0(s)
	{	\overline{F}(1-\overline{s})}.
		$$
		Consequently, the corresponding Hardy function
		$$
		Z_F(t)=
		\left(\omega H_0(\tfrac12+it)\right)^{-1/2}
		F\left(\frac12+it\right)
		$$
		is real-valued, and
		$|Z_F(t)|=|F(\frac12+it)|$.
		
		Choose $N\leq T^\alpha$, with
		$\alpha<1$, and the same resonator
		$R(t)$ as in the proof of
		Theorem~\ref{thm:linear combination selberg}.
		By the argument given there, we have
		$$
		I_F\gg
		T\exp\left(
		(\sqrt{\kappa(1-\eta)}+o(1))
		\sqrt{\frac{\log N}{\log\log N}}
		\right)\sum_{n\le N}|r(n)|^2.
		$$

Moreover, since $F$ is a finite linear combination
of the $L$-functions $F_j$, the required upper bound
follows by applying the corresponding estimates
to each $L$-functions $F_j$. If $d=1$, we proceed as in
the proof of Theorem~\ref{thm:S1}. If $d=2$, we proceed as in the
proof of Theorem~\ref{thm general}. This completes the proof.
	
\end{proof}

\begin{rem}
	An example in degree $1$ is given by linear
	combinations of distinct primitive Dirichlet
	$L$-functions having the same conductor and parity,
	since their functional-equation factors coincide.
	In this case, the conclusion of
	Theorem~\ref{thm:general linear combination hardy}
	holds with $C_F=1/2$, provided the coefficients
	satisfy \eqref{condition on functional equation}.
	
	Similarly, in degree $2$, we may consider linear
	combinations of $L$-functions associated with distinct
	normalized primitive holomorphic Hecke eigenforms of
	the same weight and level, whose nebentypus characters
	need not be the same. In this case, the conclusion
	holds with $C_F=1/\sqrt{2}$, provided the coefficients
	satisfy \eqref{condition on functional equation}.
\end{rem}

\begin{example}
	Consider
	$$
	g(z)=\Delta(z)^2,
	$$
	where $\Delta(z)$ is the Ramanujan cusp form of
	weight $12$. Then $g\in S_{24}(\mathrm{SL}_2(\mathbb Z))$,
	and its first Fourier coefficient is zero.
		Since this space has dimension two, we can write
	$$
	g=c(f_1-f_2),
	$$
	where $c$ is a nonzero real constant and
	$f_1,f_2$ are distinct normalized primitive
	holomorphic Hecke eigenforms of weight $24$
	and level one.
	Both associated $L$-functions have root number $+1$.
	Thus, the condition
	$\beta_j\omega_j=\omega\overline{\beta_j}$
	is satisfied with $\beta_1=c$, $\beta_2=-c$,
	and $\omega=1$.
	
	Consequently,
	$$
	F_g(s)=L\left(s+\frac{23}{2},g\right)
	=c\left(F_1(s)-F_2(s)\right),
	$$
	where $F_1$ and $F_2$ are the corresponding normalized $L$-functions.
	Therefore, Theorem~\ref{thm:general linear combination hardy}
	applies to $F_g$ with $C_F=1/\sqrt{2}$.
	
	This provides an example of a Dirichlet series
	satisfying a degree-two functional equation
	and the conclusion of
	Theorem~\ref{thm:general linear combination hardy},
	but not belonging to the Selberg class,
	since its first Dirichlet coefficient is zero
	and its coefficients are not multiplicative.
\end{example}

	%	\normalcolor	

		%%%%%%%%%%%%%%%%%%
		
		\section{Acknowledgments}
		
		A.M.~Hossain is supported by the Institute Fellowship from IIT Kharagpur, India.
		K.~Mahatab is  supported by the ARG-MATRICS Grant
		(Grant No.~ANRF\slash ARGM\slash 2025\slash 002540\slash MTR).

		%\newpage
		%\bibliographystyle{plain}
		%\bibliography{reference}
	\end{document}